\documentclass[12pt, twoside]{article}
\usepackage{amsmath,amssymb,amsthm,mathrsfs, verbatim,color,lscape, pifont}
\usepackage{listings}
\usepackage{graphicx}
\usepackage{secdot}
\usepackage{amsmath, dsfont}
\usepackage{amsfonts}
\usepackage{amssymb}
\usepackage[hidelinks]{hyperref}
\usepackage[english]{babel}
\usepackage{soul}
\usepackage{enumerate}
\usepackage{natbib}
\usepackage{bigstrut}
\usepackage{multirow}
\usepackage{authblk}
\usepackage[includeheadfoot,margin=0.8in]{geometry}
\usepackage{fancyhdr}
\usepackage{sectsty}
\usepackage[font=footnotesize,labelfont=bf]{caption}
\usepackage{tabularx}
\usepackage{tabu}
\usepackage{longtable}
\usepackage{subcaption}
\usepackage{pgfplots}
\usepackage{algorithm}
\usepackage{array}
\usepackage{algpseudocode}
\usepackage{tikz}
\usetikzlibrary{arrows,shapes,petri,topaths,positioning}
\usepackage{IEEEtrantools}
\newenvironment{customlegend}[1][]{%
    \begingroup
    \csname pgfplots@init@cleared@structures\endcsname
    \pgfplotsset{#1}%
}{%
    \csname pgfplots@createlegend\endcsname
    \endgroup
}%
\def\addlegendimage{\csname pgfplots@addlegendimage\endcsname} 

\newlength\figureheight 
\newlength\figurewidth 

\def\addlegendimage{\csname pgfplots@addlegendimage\endcsname}

\bibpunct{(}{)}{;}{a}{,}{,}

\allsectionsfont{\sffamily}

\newtheorem{definition}{Definition}
\newtheorem{example}{Example}

\def \-> {\rightarrow}

\def \argmin {\operatorname{argmin}}

\title{Optimal Design of Line Replaceable Units}

\author[1]{Joni Driessen, Joost de Kruijff}
\affil[1]{\footnotesize School of Industrial Engineering, Eindhoven University of Technology}
\author[2]{Joachim Arts\thanks{corresponding author, e-mail:  joachim.arts@uni.lu}}
\affil[2]{{\footnotesize Luxembourg Centre for Logistics and Supply Chain Management, University of Luxembourg}}
\author[1]{Geert-Jan van Houtum}

\newtheorem{theorem}{Theorem}
\newtheorem{lemma}[theorem]{Lemma}

\theoremstyle{remark}
\newtheorem{remark}{\em\bf Remark}

\renewcommand{\baselinestretch}{1.25}

\begin{document}
{\renewcommand{\baselinestretch}{1.25}
\maketitle

\begin{abstract}
\noindent {\footnotesize
A Line Replaceable Unit (LRU) is a collection of connected parts in a system that is replaced when any part of the LRU fails. Companies use LRUs as a mechanism to reduce downtime of systems following a failure.
The design of LRUs determines how fast a replacement is performed, so a smart design reduces replacement and downtime cost. A firm 
must purchase/repair a LRU upon failure, and large LRUs are more expensive to purchase/repair. Hence, a firm seeks to design LRUs such that the average costs per time unit are minimized. 
We formalize this problem in a new model that captures how parts in a system are connected, and how they are disassembled from the system. Our model optimizes the design of LRUs such that the replacement (and downtime) costs and LRU purchase/repair costs are minimized. We present a set partitioning formulation for which we prove a rare result: the optimal solution is integer, despite a non--integral feasible polyhedron. Secondly, we formulate our problem as a binary linear program. 
The paper concludes by numerically comparing the computation times of both formulations and illustrates the effects of various parameters on the model's outcome.
}
\end{abstract}}

\renewcommand{\baselinestretch}{1.3}

\section{Introduction}

System failures are major frustrations for their users. The consequences of failures can vary from discomfort, to disutility, to direct cost penalties such as downtime cost. In particular, industries that rely on systems to render a service or to manufacture a product experience high downtime cost. For example, the downtime cost of a computer system of a brokerage company is roughly between \$100,000--\$1,000,000 per hour \citep{cnet} and the downtime cost in the semiconductor industry is in the order of magnitude of \$100,000 per hour \citep{edn2000}. 
Together, the cost of downtime and maintenance can constitute up to 70--80\% of a system's total life cycle costs \citep{oner2007life}. As a consequence, it is crucial for a company's profitability to minimize the cost of downtime and maintenance. Significant cost reductions can be realized during the design phase of a system with relatively little effort. \citet{asiedu1998product} show that 70--85\% of the total life cycle costs are determined during the design phase, even though costs accrued in this phase only accounts for 10--20\% of the life cycle costs \citep{norman1990life,saranga2006optimization,oner2007life}. 
Hence, practice and research explore various system design concepts that enable a reduction of the after--sales costs including downtime and maintenance costs. Such concepts include common components \citep{thonemann2000optimal,briant2004optimal,driessen2017optimal}, reliability and redundancy optimization \citep{oner2013redundancy,xie2014maximizing}, and a smart design of Line Replaceable Units \citep{puig2015defining}. 

A Line Replaceable Unit (LRU) is a collection of connected system parts that can be easily replaced when one of the parts in the LRU fails. An example of such a LRU is the wheel of a car that can be quickly replaced by a spare wheel incase of a tire puncture. 
The design of LRUs has a particularly large effect on downtime and maintenance, because they directly determine long a maintenance intervention lasts \citep{dhillon1999engineering,muckstadt2004analysis,birolini2007reliability,kumar2012reliability}. It is crucial for companies to design the LRUs in a smart way that minimizes the maintenance and downtime costs. Discussions with our project partners, ASML (a manufacturer in the semiconductor industry) and Dutch Railways (a maintenance, repair and overhaul company in the railway industry), indicate that this problem is highly relevant for their industries. Similarly, two case studies \citep{puig2015defining,vangeel2018improving} at Thales (a manufacturer in the defense industry)  and a case study \citep{vandeursen2020atool} at Canon (a manufacturer of industrial printing equipment) report the same relevance, and other examples where LRUs are deliberately designed include \citet{durand2001maintainability,brasseur2012inside,KLMEM,klauke2015incremental}. 
On a more general note, the problem of designing LRUs is relevant for firms that 
are responsible for operational aspects (downtime and the maintenance) of systems. Such a firm may be an Original Equipment Manufacturer (OEM) that closes service contracts with its customers, but it may also be a firm that sells Maintenance, Repair and Overhaul (MRO) as a service to the user. Examples of OEMs that consider the design of LRUs include PACCAR or Volkswagen Group (trucking industry), Airbus (aviation industry), and the aforementioned manufacturers ASML, Thales, and Canon; while examples of MROs that consider the design of LRUs include the aforementioned Dutch Railways and Air France Industries KLM Engineering \& Maintenance (aviation industry).

LRUs are designed based on a given system design in practice. At an OEM, who designs the systems itself, the engineers design the system and subsequently they define the LRUs. The problem of defining the LRUs for a given design of the system is called LRU design. If development time allows, the OEM's design department may redesign their system design based on the outcome of a LRU design, and subsequently determine the LRUs again. A MRO does not design the system itself and thus starts with a given design of the system, when designing the LRUs. Similar to the OEM, the outcome of the LRU design may lead a MRO to modify the initial system design and derive the LRUs from this revised design. This interactive process between system design and LRU design is a powerful concept to reduce the total expected costs for an OEM or MRO. 

%
A system consists of various critical parts that are all connected to each other. As soon as one of these parts fails, the entire system is down and the company incurs downtime cost. To reduce the downtime cost, it is essential to quickly restore the system to a functioning state by replacing the failed part. Firms typically group parts in LRUs such that each LRU can be replaced quickly if any of the parts in this LRU fails. 
As designers make larger LRUs, the purchase costs typically increase, because the LRU contains more parts and thus more value. The purchase cost can refer to the purchase cost of a new LRU or to the purchase cost of a repair depending on the context. The failure of any part in a LRU triggers the failure of the entire LRU. The LRU's failure rate therefore equals the sum of the failure rates of the parts in the LRU\footnote{This holds when the part's time between failures is exponentially distributed and represents a good approximation in other circumstances.}. Larger LRU incur higher purchase cost and fail more often than smaller LRUs.
%
%
Now, the challenge is to optimally design the LRUs such that the average costs per time unit are minimized. We address this problem in this paper when a LRU design is required to partition the parts into LRUs. The first reason for considering this partition constraint comes from interviews with our project partners. If parts are partitioned into LRUs, replacement ambiguity is avoided. Replacement ambiguity is the situation in which a part fails and it is contained in more than one LRU. The engineer can then decide what LRU to replace and this leads to each system being maintained differently, which is undesirable from an asset configuration management perspective. Secondly, partitioned LRUs simplify the OEM's/MRO's production and functional testing procedures of LRUs \citep[p.~19]{vangeel2018improving}. A third reason for a partitioning constraint is that it is in line with common practice in engineering for reliability and maintenance \citet[p.~154]{birolini2007reliability}. Nevertheless, the problem can also be solved without the partition constraint as shown in the appendix. In the remainder (Section \ref{sec:numerical_experiments}) we numerically illustrate that the partition restriction leads to very limited increased costs, while it significantly increases the practical applicability of the LRU design.
 
Our LRU design problem relates to multi--component maintenance research with structural dependencies. Structural dependency between parts occurs when some parts have to be replaced or removed before the failed part can be replaced. Practitioners frequently face this type of dependency, but the academic field studying it ``is wide open \ldots and there have only been a few articles published on this topic" \citep{nicolai2008optimal}. The problem of designing LRUs naturally falls into this class, and only two papers address this problem.
\citet{thomas1986survey} poses the question whether to replace the entire car, the engine or just the piston rings in case the piston rings need replacement. More recently, \citet{puig2015defining} have revisited the question posed by \citet{thomas1986survey} and propose a model to design LRUs based on a narrow set of potential LRUs. Both works start from a bill of materials structure, i.e., the system's structure is a tree. 
The major issue with tree structure is that it does not capture the structural dependencies between parts. Tree structures cannot model the connections that exist and that have to be broken in order to replace a failed part, i.e., the connections between parts need not adhere to a tree structure. By contrast, we 
focus on the connections (cables, hoses, bolt--nut connections etc.) between various parts in a system, i.e., we consider the structural dependencies between parts. We explicitly incorporate disassembly sequences that exist for maintenance based on connections. Literature typically considers part based disassembly sequences: Part B must be disassembled before part A can be disassembled \citep{defazio1987simplified, gupta1998product,lambert2007optimizing}. This means that \emph{all} connections that part B has with other parts have to be broken before the connections of part A can be broken that enable part A's replacement. We capture these disassembly dynamics by considering the connections between parts. The connection oriented disassembly sequence has a major benefit over the part based disassembly sequence: it allows \emph{a subset} of a part's connections to be broken. For instance, if we want to replace the part A, connection based disassembly sequences allow a subset of part B's connections to be broken; whereas the part based disassembly sequence forces  all connections of part B to be broken. Only breaking a subset of connections is common in practice -- as observed at our project partners, e.g. part B can be tilted after unmounting the bolts and leaving a hose still attached to this part. Subsequently, a part A can be reached and replaced without fully disconnecting part B. 

%
%

Modelling the disassembly of a system based on its connections and the subsequent disassembly sequence is new and it enables us to accurately model the time needed to replace any LRU. Thus, we endogenize the replacement time and therewith the replacement cost of a LRU. This contrasts with \citet{puig2015defining} where the replacement time and cost are exogenously given. Furthermore, our modeling considers the full set of potential LRUs, because a LRU -- in our model -- can be any combination of parts in the system, contrasting \citet{thomas1986survey} and \citet{puig2015defining} who must pre-define all potential LRUs and their corresponding parameter values. 

Another line of related research studies LRUs (which are called modules) from a systems engineering perspective, where ``a module is a unit whose parts are powerfully connected among themselves, and relatively weakly connected to parts in other units" \citep{baldwin2000design}. This literature stream typically describes a system in terms of parts that are connected to each other, and these connections are commonly depicted in a Design Structure Matrix (DSM) \citep{steward1981design}. However, this approach neglects the disassembly sequence dynamics that exist for the replacement of LRUs (or modules). \citet{papalambros1997model} manage to relate this line of research to the area of optimization. Most research in the DSM stream aims to define measures of modularity and optimize these. Such measures typically focus on the connections between parts, and the measures prefer a high number of intra--LRU connections and a low number of inter--LRU connections; see \citet{newcomb1998implications}, \citet{sharman2004characterizing}, \citet{sosa2007network}, and \citet{wilschut2017multilevel}. Optimization of of the defined measures is typically done by using genetic algorithms \citep{meier2006design,yu2007information} or simulated annealing \citep{thebeau2001knowledge}. The aforementioned research focuses on single product DSMs, whereas \citet{alvaro2011optimal} and \citet{kim2021designing} apply this approach to a product family.

Work on (dis)assembly sequencing also has similarities to our work, because this stream models the (dis)assembly sequence that exists between parts in much detail \citep{defazio1987simplified, gupta1998product,lambert2007optimizing}. 
Research in this area optimizes the (dis)assembly sequence. We do not optimize this sequence, but we consider it to be given and focus on optimizing the design of LRUs. 

Finally, our work relates to several operations research studies that consider the impact of modular design on operations. These studies are often combinatorial in nature and aim to design product configurations such that the demand for end products is met and the average costs per time unit are minimized; see for example \citet{swaminathan1998managing}, \citet{thonemann2000optimal}, and \citet{briant2004optimal}. 
The structure in their problems superficially resembles ours, because we also study configurations of parts, which are LRUs in our case. The main difference is that we model the connections between parts and the disassembly sequences that exist for maintenance, while research in this stream does not.

In this paper, we make the following contributions: we present (i) a novel way to represent a system with multiple parts that are connected to each other, and we incorporate the disassembly sequences that exist for maintenance based on connections rather than parts. Modeling the connections and disassembly sequences enables us to endogenize the downtime cost due to the replacement of a LRU containing the failed part. Next, we use our system description to define an optimization model -- called \textsc{LRU Design} -- that minimizes the sum of the replacement and purchase cost by optimizing the LRU designs. 

We provide (ii) a set partitioning formulation of \textsc{LRU Design} that allows for branch--and--price algorithms. 
Next, we prove (iii) that an optimal solution to the set partitioning formulation is integer. This result is rather remarkable, 
because the feasible polyhedron is not integral. There exist two problems that also posses this property: a minimax transportation problem \citep{ahuja1986algorithms} and a multi-period machine assignment problem \citep{zhang2006multi}. The majority of research typically shows the existence of an optimal integer solution by proving that the feasible polyhedron is integral (e.g. through total unimodularity), see for instance \citet{hillier1998optimal}, \citet{ball2003stochastic}, 
\citet{churchill2012determining}, \citet{gamvros2012multi}. 
Our integrality result cannot be established in this way due to the non--integral polyhedron. We define a so--called LRU cycle and prove that a solution that contains such a LRU cycle is suboptimal.
Subsequently, we study the matrix encoding of an optimal solution to prove that an optimal solution is integer.  
We believe that our proof approach is applicable and promising to other problems that can be formulated as set partitioning problems, because one's main effort would be to prove suboptimality of partitions that contain cycles. 

Fourth (iv), we focus on additive failure rates such that we obtain linear expressions that support the implementation of our models. We specify the set partitioning formulation under such additive failure rates and we formulate \textsc{LRU Design} as a binary non--linear program (BNLP), which we then transform into a binary linear program.

Finally, (v) we illustrate that the set partitioning formulation is suitable for large instances, and we study the effects of various parameters on the model's outcome. Moreover, we numerically show that the relative cost impact of introducing a partitioning constraint is very limited, while we strongly improve practical applicability of our model. 

The rest of this paper is organized as follows. In Section \ref{sec:model_description}, we discuss our system representation including the disassembly sequences, and the optimization model \textsc{LRU Design}. We present a set partitioning formulation of \textsc{LRU Design} in Section \ref{sec:analysis_cg}, and we prove that an optimal solution is integer for this formulation. In the succeeding part of the paper, we focus on additive failure rates (for implementation convenience) and we present the set partitioning formulation under this assumption. Moreover, we discuss a binary non--linear programming (BNLP) formulation of \textsc{LRU Design} in Section \ref{sec:bp}, which we subsequently linearize to obtain a binary linear program (BLP). Finally in Section \ref{sec:numerical_experiments}, we numerically compare the computation times of the BLP formulation to the set partitioning formulation, we illustrate the effects of various parameter perturbations on the model's outcome, and illustrate the limited cost impact of the partitioning constraint in our problem. We offer concluding remarks in Section \ref{sec:conclusions}. 

\section{Model}
\label{sec:model_description}

First this section explains how a system with parts an connections can be represented by means of an example. The generalization then follows from the example and we present the optimization model called \textsc{LRU Design}.

\subsection{An illustrative example}
Consider a laptop repair shop that repairs laptops by removing failed parts from the laptop and replacing the failed parts with new ones.
The repair shop's objective is to design the LRUs such that it minimizes the cost of repair time and procurement of new LRUs.
We consider the illustrative example of a laptop, because this system is technologically simple and many people have some familiarity with it. Bear in mind however that the model was designed for and has greater financial impact for large and technologically complex systems. Unfortunately, such a system lacks the familiarity of general readership. 

The laptop example is based on data for a Dell Precision 7710 laptop \citep{dell_site}. Each part has a purchase cost and a failure rate (in failures per year). Table \ref{tab:laptop_parts} lists the estimated purchase cost and fictitious failure rate for all parts (in failures per year). The purchase cost of a part is its price found online on websites such as amazon.com.

\begin{table}[htbp!]
\centering
\begin{tabular}{llrr}
\hline
Identifier	&	Part Name & Part Cost (\$)	& Failure rate (failures/year)	\bigstrut[t]\\
\hline
A	&	Battery	&	180	&	0.3\bigstrut[t]\\
B	&	Hard Disk Drive	&	170	&	0.2\\
C	&	Keyboard 	&	45	&	0.001\\
D	&	WLAN Card 	&	50	&	0.15\\
E	&	Palm Rest 	&	45	&	0.001\\
F	&	Speakers 	&	14	&	0.05\\
G	&	Heat Sink	&	75	&	0.1 \\
H	&	4 GB Video Card 	&	250	&	0.1\\
I	&	Display Housing 	&	40	& 0.001\\
J	&	Display Front Cover 	&	20	& 0.001 \\
K	&	Display Bezel 	&	170	& 0.25 \\
L	&	Motherboard	&	270	& 0.25 \\
M	&	Computer Base	&	50	&	0.001\\
\hline
\end{tabular}
\caption{Part identifier list}
\label{tab:laptop_parts}
\end{table}
Each of the parts is connected to other parts, e.g. the Palm Rest is screwed to the Computer Base, the Palm Rest is wired to the Motherboard, and the Palm Rest is screwed to the Keyboard. Thus, there exist connections $\{E,M\}$, $\{E,L\}$, and $\{C,E\}$. 
In the event the Palm Rest fails and one wishes to replace it individually, one has to break all the connections that the Palm Rest has with all other parts: $\{E,M\}$, $\{E,L\}$ and $\{C,E\}$. Breaking each connection takes a certain amount of time, which translates into costs by multiplying the time with a cost rate, e.g. the salary rate of the repair man or downtime penalty. When the failed Palm Rest has been disconnected from the system, a new and identical Palm Rest from stock is installed into the system by reconnecting all the connections that have been broken previously (in order to remove the failed Palm Rest). This re--establishing of connections also costs time and can be translated into costs as well. 
Finally, a new Palm Rest is purchased to replenish the stock.

All information about parts, connections, failure rates, purchase costs, and the costs of breaking and re--establishing connections can be represented in a weighted un-directed graph. The parts correspond to vertices, and the part connections correspond to the edges. Furthermore, the failure rates and the purchase costs are attributes of the vertices, and the costs for breaking and re--establishing a connection correspond to the weight of an edge in the graph. For the laptop example, this graph can be found by analyzing the Owner's manual \citep{dell_site}, and is given in Figure \ref{fig:connection_graph}. The cost of breaking and re--establishing a connection is an estimate and is depicted on the edges.

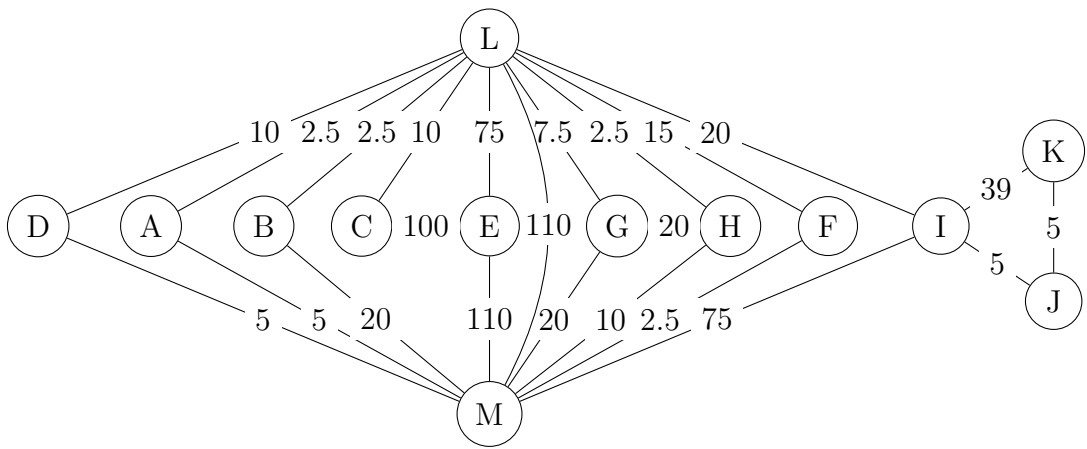
\begin{figure}[htbp!]
\centering
\begin{tikzpicture}[scale=1,transform shape]
\tikzset{VertexStyle/.style = {circle,draw,fill=white,minimum size=0.75cm}}
\tikzset{LabelStyle/.style = {draw=none,fill=white,text=black}}
\tikzset{EdgeStyle/.style   = {draw,fill=none}}
	\node[VertexStyle] (L) at (0,4.5) {L};
	\node[VertexStyle] (D) at (-6,2) {D};
	\node[VertexStyle] (A) at (-4.5,2) {A};
	\node[VertexStyle] (B) at (-3,2) {B};
	\node[VertexStyle] (C) at (-1.7,2) {C};
	\node[VertexStyle] (G) at (1.7,2) {G};
	\node[VertexStyle] (H) at (3.2,2) {H};
	\node[VertexStyle] (F) at (4.5,2) {F};
	\node[VertexStyle] (M) at (0,-0.5) {M};
	\node[VertexStyle] (I) at (6,2) {I};
	\node[VertexStyle] (K) at (7.5,3) {K};
	\node[VertexStyle] (J) at (7.5,1) {J};

	\draw[EdgeStyle] (G) to node[LabelStyle]{7.5} (L);
	\draw[EdgeStyle] (G) to node[LabelStyle]{20} (M);
	\tikzset{EdgeStyle/.append style = {bend left=28}}
	\draw[EdgeStyle] (L) to node[LabelStyle]{110} (M);
	\node[VertexStyle] (E) at (0,2) {E};	
	
	\tikzset{EdgeStyle/.style   = {draw,fill=none}}
	\draw[EdgeStyle] (A) to node[LabelStyle]{2.5} (L);
	\draw[EdgeStyle] (B) to node[LabelStyle]{2.5} (L);
	\draw[EdgeStyle] (C) to node[LabelStyle]{10} (L);
	\draw[EdgeStyle] (D) to node[LabelStyle]{10} (L);
	\draw[EdgeStyle] (E) to node[LabelStyle]{75} (L);
	\draw[EdgeStyle] (F) to node[LabelStyle]{15} (L);
	\draw[EdgeStyle] (H) to node[LabelStyle]{2.5} (L);
	\draw[EdgeStyle] (A) to node[LabelStyle]{5} (M);
	\draw[EdgeStyle] (B) to node[LabelStyle]{20} (M);
	\draw[EdgeStyle] (D) to node[LabelStyle]{5} (M);
	\draw[EdgeStyle] (M) to node[LabelStyle]{110} (E);
	\draw[EdgeStyle] (F) to node[LabelStyle]{2.5} (M);
	\draw[EdgeStyle] (H) to node[LabelStyle]{10} (M);
	
	\draw[EdgeStyle] (C) to node[LabelStyle]{100} (E);
	\draw[EdgeStyle] (G) to node[LabelStyle]{20} (H);
	
	\draw[EdgeStyle] (I) to node[LabelStyle]{75} (M);
	\draw[EdgeStyle] (I) to node[LabelStyle]{5} (J);
	\draw[EdgeStyle] (J) to node[LabelStyle]{5} (K);
	\draw[EdgeStyle] (I) to node[LabelStyle]{39} (K);
%
	\draw[EdgeStyle] (L) to node [LabelStyle]{20} (I);
\end{tikzpicture}
\caption{The laptop's connection graph}
\label{fig:connection_graph}
\end{figure}

We call the graph in Figure \ref{fig:connection_graph} the \emph{connection graph}. The connection graph may suggest that we only need to break connections $\{E,L\}$, $\{E,M\}$ and $\{C,E\}$ in order to remove the Palm Rest. 
However, the Owner's Manual states that in order to disconnect the Palm Rest, one must first break the connections that enable one to remove the Keyboard (C), the Hard Disk Drive (B), and the Battery (A); i.e., there is a disassembly sequence. This implies that there exists a collection of connections that needs to be broken prior to breaking the connections $\{E,L\}$, $\{E,M\}$, or $\{C,E\}$ \citep{dell_site}. 
Therefore there is a predecessor--successor relationship for breaking (and re--establishing) the connections depicted in Figure \ref{fig:connection_graph}. 
We model such predecessor--successor relationships in a separate directed graph, which we call the \emph{precedence graph}. An arc in the precedence graph from an edge $\{E,M\}$ to $\{E,L\}$ implies that connection $\{E,L\}$ must be broken before connection $\{E,M\}$ can be broken. Figure \ref{fig:precedence_graph} shows the precedence graph for the laptop \citep{dell_site}.
\begin{figure}[htbp!]
\centering
\begin{tikzpicture}[scale=1,transform shape]
\tikzset{VertexStyle/.style = {circle,draw,fill=white,minimum size=1.5cm}}
\tikzset{LabelStyle/.style = {draw=none,fill=none}}
\tikzset{EdgeStyle/.style   = {->,>=stealth, draw,fill=none}}
	\node[VertexStyle] (1) at (-3,6) {\{{D,L}\}};
	\node[VertexStyle] (2) at (-5,4) {\{{A,L}\}};
	\node[VertexStyle] (3) at (-3,4) {\{{B,L}\}};
	\node[VertexStyle] (4) at (-1,4) {\{{C,L}\}};
	\node[VertexStyle] (5) at (1,4) {\{{E,L}\}};
	\node[VertexStyle] (6) at (3,4) {\{{G,L}\}};
	\node[VertexStyle] (7) at (5,4) {\{{H,L}\}};
	\node[VertexStyle] (8) at (7,4) {\{{F,L}\}};
	\node[VertexStyle] (9) at (-1,0) {\{{C,E}\}};
	\node[VertexStyle] (10) at (3,2) {\{{G,H}\}};
	\node[VertexStyle] (11) at (-5,6) {\{{D,M}\}};
	\node[VertexStyle] (12) at (-5,0) {\{{A,M}\}};
	\node[VertexStyle] (13) at (-3,0) {\{{B,M}\}};
	\node[VertexStyle] (14) at (1,0) {\{{E,M}\}};
	\node[VertexStyle] (15) at (3,0) {\{{G,M}\}};
	\node[VertexStyle] (16) at (5,0) {\{{H,M}\}};
	\node[VertexStyle] (17) at (7,0) {\{{F,M}\}};
	\node[VertexStyle] (18) at (1,-2) {\{{I,M}\}};
	\node[VertexStyle] (19) at (5,-2) {\{{K,J}\}};
	\node[VertexStyle] (20) at (3,-2) {\{{I,J}\}};
	\node[VertexStyle] (21) at (7,-2) {\{{I,K}\}};
	\node[VertexStyle] (22) at (-1,-2) {\{{I,L}\}};
	\node[VertexStyle] (23) at (9,4) {\{{L,M}\}};

	\draw[EdgeStyle] (2) to node[LabelStyle]{} (12);
	\draw[EdgeStyle] (3) to node[LabelStyle]{} (2);
	\draw[EdgeStyle] (3) to node[LabelStyle]{} (13);
	\draw[EdgeStyle] (9) to node[LabelStyle]{} (4);
	\draw[EdgeStyle] (4) to node[LabelStyle]{} (3);
	\draw[EdgeStyle] (1) to node[LabelStyle]{} (11);
	\draw[EdgeStyle] (14) to node[LabelStyle]{} (5);
	\draw[EdgeStyle] (5) to node[LabelStyle]{} (9);
	\draw[EdgeStyle] (17) to node[LabelStyle]{} (8);
	\draw[EdgeStyle] (6) to node[LabelStyle]{} (10);
	\draw[EdgeStyle] (10) to node[LabelStyle]{} (15);
	\draw[EdgeStyle] (15) to node[LabelStyle]{} (14);
	\draw[EdgeStyle] (7) to node[LabelStyle]{} (16);
	\draw[EdgeStyle] (7) to node[LabelStyle]{} (6);
	\draw[EdgeStyle] (19) to node[LabelStyle]{} (20);
	\draw[EdgeStyle] (21) to node[LabelStyle]{} (19);
	\draw[EdgeStyle] (20) to node[LabelStyle]{} (18);
	\draw[EdgeStyle] (18) to node[LabelStyle]{} (22);
	\draw[EdgeStyle] (23) to node[LabelStyle]{} (17);
	\draw[EdgeStyle] (18) to node[LabelStyle]{} (14);
	\draw[EdgeStyle] (1) to node[LabelStyle]{} (3);
	
	\tikzset{EdgeStyle/.append style = {bend left}}
	\draw[EdgeStyle] (17) to node[LabelStyle]{} (14);
	
	\tikzset{EdgeStyle/.append style = {bend right=25}}
	\draw[EdgeStyle] (23) to node[LabelStyle]{} (1);
	\tikzset{EdgeStyle/.append style = {bend right=35}}
	\draw[EdgeStyle] (23) to node[LabelStyle]{} (7);
\end{tikzpicture}
\caption{The laptop's precedence graph}
\label{fig:precedence_graph}
\end{figure}

The combination of the precedence graph (Figure \ref{fig:precedence_graph}) with the connection graph (Figure \ref{fig:connection_graph}), enables us to list \textit{all} connections that need to be broken for the replacement of an arbitrary part. For example, replacement the Palm Rest requires one to break $\{E,M\}$, $\{E,L\}$, and $\{C,E\}$ (see Figure \ref{fig:connection_graph}), but to break connection $\{C,E\}$ one must first break the set of connections $\{\{A,L\}, \{A,M\},\{B,M\},\{B,L\},\{C,L\}\}$ (see Figure \ref{fig:precedence_graph}). Similarly, one can determine all connections that need to be broken prior to $\{E,M\}$ and $\{E,L\}$. Finally, one must break all connections $\{\{A,L\}, \{A,M\},\{B,M\},\{B,L\},\{C,E\},\{C,L\},\{E,L\},\{E,M\}\}$ in order to remove the Palm Rest (E). Analogously, one must break connections $\{\{A,L\}, \{A,M\},\{B,M\},\allowbreak\{B,L\},\{C,E\},\{C,L\}\}$ to replace the Keyboard.

If one decides to replace the Palm Rest (E) together with the Keyboard (C), i.e., define a LRU $Q$ that contains C and E. However, this implies that the engineer has to break all connections $\{\{A,L\}, \{A,M\},\{B,M\},\{B,L\},\{C,L\},\{C,E\},\{E,L\},\allowbreak\{E,M\}\}$ upon the failure of either the Palm Rest (E) or the Keyboard (C). As a consequence, one must break the expensive edges $\{E,L\}$ and $\{E,M\}$ more often than when the Palm Rest and the Keyboard are separate LRUs. Furthermore, the LRU $Q$ has a higher purchase cost as well as a higher failure rate compared to the Palm Rest and the Keyboard individually. Therefore, it is better to keep the Palm Rest and the Keyboard as separate LRUs instead of combining these two into one LRU $Q$. It is now of interest to find the optimal design of LRUs that minimizes the sum of the replacement and purchase costs, based on the connection graph in Figure \ref{fig:connection_graph} and the precedence graph in Figure \ref{fig:precedence_graph}. 

\subsection{A generic model}
\label{sec:optimization_model}
The example above illustrates how a system is built up and what relationships parts and connections have. The approach used for the laptop also applies to more complicated systems such as a bogie in a train (Bombardier/Dutch Railways), a positioning module in a lithography system (ASML), a truck engine (PACCAR/Volkswagen Group), or a jet engine (Pratt \& Whitney). Consider a system that consists of multiple parts, and assume that maintenance is done upon the failure of a part. 
Moreover, assume that one can accurately and instantaneously determine which part has failed, when the system fails as a whole. 
The system is defined by two graphs: a weighted undirected connection graph $G$ and a directed precedence graph $D$. 
The graph $G=(V,E)$ is characterized by the set of vertices $V$ and the set of edges $E$. The former set $V$ corresponds to the parts in the system, and the latter set $E$ corresponds to the connections between parts. Furthermore, each part in $G$ has a purchase cost $\ell:V\rightarrow \mathbb{R}_+$, where $\mathbb{R}_+=\{x\in\mathbb{R} \,|\, x>0\}$, and the failure rate for any subset $Q\subseteq V$ is given by $\lambda:2^{V}\rightarrow\mathbb{R}_+$. We also assume that the failure rate function is superadditive, i.e., $\lambda(R)+\lambda(T)\leq \lambda(R\cup T), \, \forall R,T\subseteq V, R\cap T =\emptyset$. 
The cost to break a connection are given by the edge costs $w:E\rightarrow \mathbb{R}_+$. We use the terms part and vertex interchangeably, as well as the terms connection and edge. At the end of this section, we discuss 
what happens when LRUs are repaired rather than purchased, see Remark \ref{rem:repair_LRU}.

Besides the connection graph $G$, the precedence graph $D=(E,A)$ is an unweighted acyclic directed graph that captures the disassembly sequences of the connections $e\in E$. The set $A$ corresponds to the set of arcs, and an arc $(i,j)\in A$ from edge $i$ to edge $j$ exists if and only if edge $j$ has to be broken \textit{before} edge $i$ can be broken. 
We assume that arcs can only connect adjacent edges, i.e., all arcs in $A$ satisfy $(\{u,v\},\{v,x\})\in A: u,v,x\in V$ and $u\neq v \neq x$. 
The graph $D$ determines a set $H(e)$ of successor edges for each edge $e\in E$. This set $H(e)$ consists of all edges including the edge $e\in E$ that must be disconnected in order to break $e$, and it can be determined by using the polynomial time Algorithms \ref{alg:derive_H} and \ref{alg:derive_degree} in Appendix \ref{app:algorithms}. We remark that $H(e)$ is a directed tree rooted at $e\in E$.

Further, we assume that $G$ is connected, without loss of generality. If $G$ is not connected, there do not exist arcs $(\{u,v\},\{x,y\})\in A$ such that $u,v$ are in one connected component and $x,y$ are in the other connected component. This follows because all arcs in $A$ satisfy $(\{u,v\},\{v,x\})\in A:u,v,x\in V$ and $u\neq v\neq x$. Hence, if $G$ were disconnected, we apply our model to each connected component of $G$ with the precedence graph induced by the connected component. 
We define a LRU design as partition $S$ of the vertices $V$. Connections need to be broken in order to replace a LRU $Q\in S$ from the system. First, define the set $B(Q)=\{\{u,v\}\in E: u\in Q, v\in V\setminus Q\}$, as the set of all edges that connect the LRU to the other parts of the system not in the LRU. That is, the set $B(Q)$ contains the edges that cross the LRU's boundary. Next, for the removal of LRU $Q$, one must break all the edges $e\in B(Q)$, as well as all the edges that need to be broken prior to breaking any edge $e\in B(Q)$. Hence, $\Gamma(Q)=\bigcup_{e\in B(Q)}H(e)$ is the set of edges that need to be broken in order to replace a LRU $Q\in S$.
Note that $\Gamma(Q)$ may contain edges between vertices in $Q$. This is a model feature as it allows us to model LRUs such as a chain between two cogwheels. If each each link in a chain is a vertex that is connected to the adjacent vertices, then it is necessary to break the edge between two links in order to remove the chain from the cogwheels. A detailed example of this is provided in Appendix \ref{sec:app:chain}. 

Each LRU $Q$ has a purchase cost and failure rate. The purchase cost of a LRU is given by the sum of the purchase cost of all parts in the LRU, i.e., the LRU's purchase cost is given by $\sum_{v\in Q}\ell(v)$. We relax this assumption in Remark \ref{rem:repair_LRU}.  
The total failure rate of a LRU $Q\in S$ is denoted by $\lambda(Q)$. 

Next, we derive the cost expression for a LRU $Q\in S$. Upon the failure of LRU $Q$, one breaks all edges $e\in \Gamma(Q)$ resulting in the cost $\sum_{e\in \Gamma(Q)}w(e)$. Moreover, replacement LRU is purchased at cost $\sum_{v\in Q}\ell(v)$. The average cost per time unit of LRU $Q$ then satisfies
\begin{equation}
\omega(Q)=\lambda(Q)\left(\sum_{e\in \Gamma(Q)}w(e)+\sum_{u\in Q}\ell(u)\right).
\label{eq:cost_LRU}
\end{equation}

As a LRU design $S$ is a partition of $V$ and $Q\in S$, the total cost per time unit of a LRU design is given by
\begin{equation}
\begin{split}
\pi(S)&=\sum_{Q\in S}\omega(Q).
\end{split}
\label{eq:cost_partition}
\end{equation}

The \textsc{LRU Design} problem can now be simply stated as: What is the LRU Design $S$ that minimizes $\pi(S)$?

\textsc{LRU Design} has the property that each LRU $Q$ in optimal solution $S^*$ to \textsc{LRU Design} is a connected subgraph of $G$. 

\begin{lemma}
Each LRU $Q\in S^*$ is a connected subgraph of $G$, for any optimal solution $S^*$ to \textsc{LRU Design}. 
\label{lem:connected_LRU}
\end{lemma}
\begin{proof}
Let $S^*$ be an optimal solution to \textsc{LRU Design}, and let $\mathscr{J}$ be the finite set of connected components in the subgraph induced by a LRU $Q\in S^*$. The set $\mathscr{J}$ partitions $Q$, $\mathscr{J}$ is finite because $Q$ is finite, and $|\mathscr{J}|\geq 1$. The case of $|\mathscr{J}|=1$ implies that $Q$ is connected, which satisfies our claim. Thus, we consider the case $|\mathscr{J}|\geq 2$ in the remainder, and observe that in this case $\nexists \{u,v\}\in E: u\in \mathcal{J}_1,v\in \mathcal{J}_2$ with $\mathcal{J}_1,\mathcal{J}_2\in\mathscr{J}$. Note further that $\Gamma(\mathcal{J})\subseteq \Gamma(Q)$ for a $\mathcal{J}\in\mathscr{J}$ as $\mathcal{J}\subset Q$. Thus, 
\begin{IEEEeqnarray*}{rCl}
\sum_{\mathcal{J}\in\mathscr{J}}\omega(\mathcal{J})	&=& \sum_{\mathcal{J}\in\mathscr{J}}\lambda(\mathcal{J})\sum_{e\in \Gamma(\mathcal{J})}w(e)+\sum_{\mathcal{J}\in\mathscr{J}}\lambda(\mathcal{J})\sum_{u\in \mathcal{J}}\ell(u)\\
						&\leq &\sum_{\mathcal{J}\in\mathscr{J}} \lambda(\mathcal{J})\sum_{e\in \Gamma(Q)}w(e) +\sum_{\mathcal{J}\in\mathscr{J}}\lambda(\mathcal{J})\sum_{u\in \mathcal{J}}\ell(u) \\
						&< &\sum_{\mathcal{J}\in\mathscr{J}} \lambda(\mathcal{J})\sum_{e\in \Gamma(Q)}w(e) +\sum_{\mathcal{J}\in\mathscr{J}}\lambda(\mathcal{J})\sum_{u\in Q}\ell(u) \\
						&\leq&\lambda(Q)\sum_{e\in \Gamma(Q)}w(e)+\lambda(Q)\sum_{u\in Q}\ell (u)=\omega(Q),
\end{IEEEeqnarray*}
where the first inequality follows from the fact that each $\Gamma(\mathcal{J})\subseteq \Gamma(Q),\; \forall \mathcal{J}\in\mathscr{J}$. 
The second inequality follows from the fact that $\mathcal{J}\subset Q,\; \forall \mathcal{J}\in\mathscr{J}$, and thus $\sum_{u\in \mathcal{J}}\ell(u)< \sum_{u\in Q}\ell(u),\; \forall \mathcal{J}\in\mathscr{J}$. The last inequality holds because $\mathscr{J}$ partitions $Q$ and the failure rate function $\lambda$ is superadditive. 
Hence, we obtain $\sum_{\mathcal{J}\in\mathscr{J}}\omega(\mathcal{J})<\omega(Q)$, where $\mathscr{J}$ partitions $Q$. However, this contradicts the optimality of $S^*$. Therefore, each LRU $Q\in S^*$ is a connected subgraph of $G$, for any optimal solution $S^*$ to \textsc{LRU Design}.  
\end{proof}

We will use Lemma \ref{lem:connected_LRU} throughout this paper. We conclude this section with two remarks that enable further generalization of \textsc{LRU Design}.
%
%
%
%
\begin{remark}
We assumed that we do not repair a LRU, and thus purchase a new one. If we relax this assumption and repair a failed part of a LRU offline, we incur a total repair cost per time unit of $\lambda(V)\sum_{v\in V}q(v)$, where $q(v)$ is the repair cost of part $v$. However, if we repair a part of a LRU, we have to test the entire LRU to see whether it functions again. This means that we have to test each part in the LRU, and thus $\ell(v)$ now represents the cost of testing part $v\in V$ offline. Larger LRUs, now, have more parts that need to be tested before the LRU is certified as repaired. The total repair cost per time unit is sunk as $\lambda(V)\sum_{v\in V}q(v)$ is independent of the LRU design, but we still have the testing cost per time unit of LRU $Q$ given by $\lambda(Q)\sum_{u\in Q}\ell(u)$. Hence, our model \textsc{LRU Design} still applies. 
%
%
%
%
%
\label{rem:repair_LRU}
\end{remark}

\section{Set partitioning formulation}
\label{sec:analysis_cg}
We formulate \textsc{LRU Design} as a set partitioning problem that allows for column generation (branch--and--price) algorithms. Then, we prove in Section \ref{sec:relationship_M_LPM} that an optimal solution to the relaxed master program is integer, even though the feasible polyhedron is \emph{not} integral. Finally, we present the column generating procedure in Section \ref{sec:solving_LPM} for solving the set partitioning formulation of \textsc{LRU Design}.

A LRU design $S$ consists of various non--intersecting LRUs $Q\in S$ that have been selected. Let $\mathscr{S}=2^V$ be the power set of $V$ from which LRUs can be selected; $\mathscr{S}$ contains all possible LRUs. Then, $S\subset \mathscr{S}$, and our objective is to determine which solution $S$ is optimal via column generation. A LRU $Q\in\mathscr{S}$ can equivalently be represented as a $(0,1)$ column with $|V|$ elements, where a $1$ indicates that a vertex is in the LRU $Q$ and a $0$ denotes that the vertex does not belong to the LRU $Q$. Hence, we consider the matrix entries $z_{vQ}$ that equal $1$ if $v\in Q$ and $0$ otherwise. Then, a column from the matrix $Z=(z_{vQ})$ corresponds to LRU $Q$, and we denote this column by $Z_Q$. Note that a column $Z_Q$ and the LRU $Q\subseteq V$ are equivalent representations of a LRU.

Let $x_Q$ be the indicator variable that denotes whether a LRU $Q\in \mathscr{S}$ is selected for the LRU design $S$. 
We denote $\boldsymbol{x}$ as the vector consisting of all entries $x_Q$. Given $\boldsymbol{x}$,
we can straightforwardly derive the solution $S$ to \textsc{LRU Design} by $S=\{Q\in\mathscr{S}:x_Q>0\}$. 
We remark that $S$ can equivalently be represented as the submatrix $\mathcal{Z}=\{Z_Q:x_Q>0\}$ of $Z$. 
Our objective is to determine the LRU design in terms of $x_Q$ such that the average costs per time unit are minimized, and each part $v\in V$ is included in exactly one LRU. We capture this in the Master Problem (M):
%
%
%
%
\begin{IEEEeqnarray}{lllllll}
\interdisplaylinepenalty=0
\IEEEyesnumber\label{eq:mp1} \IEEEyessubnumber*
\mbox{(M)}	\qquad		&& \min_{\boldsymbol{x}}			&\quad	& \sum_{Q\in \mathscr{S}}\omega(Q)x_{Q}			&&\label{eq:mp1_obj}\\
						&& \mbox{s.t.}		&		& \sum_{Q \in \mathscr{S}}z_{vQ}x_{Q}= 1,		&\qquad &\forall v\in V, \label{eq:mp1_con_part}\\
						&&					&		& x_Q \in \{0,1\},								&&\forall Q\in \mathscr{S} \label{eq:mp1_con_integral}.
\end{IEEEeqnarray}
Recalling that $\omega(Q)$ is the average cost per time unit of using LRU $Q$, the objective function \eqref{eq:mp1_obj} minimizes the costs of using the selected LRUs, while constraints \eqref{eq:mp1_con_part} enforce that each part $v\in V$ is included in exactly one LRU $Q\in \mathscr{S}$. 
The set $\mathscr{S}$ is exponentially large, so straightforward optimization is not tractable. Therefore, we propose to solve the LP relaxation of M by column generation. We relax the integrality of $x_Q$ to obtain the LP relaxation of the Master Problem called LPM: 
\begin{IEEEeqnarray}{lllllll}
\interdisplaylinepenalty=0
\IEEEyesnumber\label{eq:mp1lpm} \IEEEyessubnumber*
\mbox{(LPM)}	\qquad		&& \min_{\boldsymbol{x}}			&\quad	& \sum_{Q\in \mathscr{S}}\omega(Q)x_{Q}			&&\label{eq:mp1_objlpm}\\
						&& \mbox{s.t.}		&		& \sum_{Q \in \mathscr{S}}z_{vQ}x_{Q}= 1,		&\qquad &\forall v\in V, \label{eq:mp1_con_partlpm}\\
						&&					&		& 0 \leq x_Q \leq 1 ,								&&\forall Q\in \mathscr{S} \label{eq:mp1_con_integrallpm}.
\end{IEEEeqnarray}
%

Subsequently, we present our procedure for solving LPM in Section \ref{sec:solving_LPM}.

\subsection{Integrality and polyhedral structure of LPM}
\label{sec:relationship_M_LPM}
We prove that an optimal solution to LPM is integer by considering a so--called LRU cycle. We show that if a given fractional solution contains a LRU cycle, there exists a feasible solution to LPM without the LRU cycle and strictly lower costs. This implies that an optimal solution does not contain a LRU cycle. 

Let $\tilde{x}$ be a fractional solution to LPM with $\tilde{S}=\{Q\in\mathscr{S}:\tilde{x}_Q>0\}$ (or equivalently $\tilde{\mathcal{Z}}=\{Z_Q:\tilde{x}_Q>0\}$) and such that each $Q\in\tilde{S}$ is a connected subgraph of $G$. Furthermore, let $x^*$ be an optimal solution to LPM with $S^*=\{Q\in\mathscr{S}:x^*_Q>0\}$ (or equivalently $\mathcal{Z}^*=\{Z_Q:x^*_Q>0\}$) and that also has connected LRUs $Q\in S^*$. Note that $x^*$ exists by Lemma \ref{lem:connected_LRU}.

\begin{definition}
A LRU cycle is a collection of LRUs $C=\{Q_1,Q_2,\ldots,Q_n\}$ such that each $Q_i$ is connected, 
$n \geq 3$ 
and for all $1\leq i\leq n$ we have $Q_i\cap Q_{i+1}\neq\emptyset$, 
$(Q_i\cap Q_{i+1})\setminus (Q_{i+1}\cap Q_{i+2})\neq\emptyset$, 
$(Q_{i+1}\cap Q_{i+2})\setminus (Q_i\cap Q_{i+1})\neq\emptyset$, 
with $n+1\equiv 1 \pmod{n}$ and $n+2\equiv 2 \pmod{n}$.
\end{definition}

%
%
For an example of a LRU cycle, we refer the reader to Figure \ref{fig:illustrate_lemma3}. We remark that there can exist a solution containing a LRU cycle such that the solution is an extreme point of the feasible polyhedron of LPM, and thus the feasible polyhedron of LPM is not integral.
Next, we prove that an optimal solution to LPM does not contain a LRU cycle: \ref{thm:intersection_cycle}.

\begin{theorem}
An optimal solution $x^*$ to LPM does not contain a LRU cycle.
\label{thm:intersection_cycle}
\end{theorem}

The proof of this theorem uses a technical lemma that can be found in the appendix \ref{app:lemma_removal_path} as Lemma \ref{lem:removal_path} and the following two concepts. Define the set of edges that is broken for a LRU $X$ but not for a LRU $Y$ by $\mathcal{F}(X,Y)=\Gamma(X)\setminus\Gamma(Y)$. Furthermore, modular arithmetic is used for the indices of LRUs that form a cycle $Q_i$, $Q_{i-1}$, and $Q_{i+1}$ with $1\leq i\leq n$, $n+1\equiv 1 \pmod{n}$, and $Q_0\equiv Q_n$.

\proof{}
We show that a solution to LPM that contains a LRU cycle is suboptimal. Let $\tilde{x}$ be a solution to LPM such that each $Q\in\tilde{S}$ is connected and there exists a LRU cycle $C=\{Q_1,Q_2,\ldots,Q_n\}\subseteq\tilde{S}$ with minimal $n$. Note that a solution that contains a LRU cycle must be fractional. We prove that there exists a feasible solution $x^{\prime}$ to LPM with $S^{\prime}=\{Q\in\mathscr{S}:x^{\prime}_Q>0\}$ in which the LRU cycle $C$ does not exist and $\pi(S^{\prime})< \pi(\tilde{S})$. 

%
%
%
%
%
Let $W_j=\min\left\{\sum_{e\in \mathcal{F}(Q_j\cap Q_{j+1},Q_j)}w(e),\sum_{e\in \mathcal{F}(Q_j\cap Q_{j-1},Q_j)}w(e)\right\}$ for each LRU $Q_j$. Next, we consider a specific LRU $Q_i=\argmin_{Q_j\in C}\{W_j\}$ and we assume that $W_i=\sum_{e\in \mathcal{F}(Q_i\cap Q_{i+1},Q_i)}w(e)$ (later we consider $W_i=\sum_{e\in \mathcal{F}(Q_i\cap Q_{i-1},Q_i)}w(e)$). 
We create an alternative solution $x^{\prime}$ by partitioning $Q_i$ in $Q_i\cap Q_{i+1}$ and $Q_i\setminus Q_{i+1}$. That is, let the alternative solution $x^{\prime}$ be identical to $\tilde{x}$ except for the entries $x_{Q_i}^{\prime}=0$, $x_{Q_i\setminus Q_{i+1}}^{\prime}=\tilde{x}_{Q_i}$, and $x_{Q_i \cap Q_{i+1}}^{\prime}=\tilde{x}_{Q_i}$. We have
\begin{IEEEeqnarray*}{rCll}
\allowdisplaybreaks
\pi(S^{\prime})-\pi(\tilde{S})	&\leq&& \tilde{x}_{Q_i} \left(\lambda(Q_i \cap Q_{i+1})\sum_{e\in\Gamma(Q_i \cap Q_{i+1})}w(e)+\lambda(Q_i \cap Q_{i+1})\sum_{u\in Q_i\cap Q_{i+1}}\ell(u)  \right. \\
								&&&+\left. \lambda(Q_i \setminus Q_{i+1})\sum_{e\in\Gamma(Q_i \setminus Q_{i+1})}w(e)+\lambda(Q_i \setminus Q_{i+1})\sum_{u\in Q_i\setminus Q_{i+1}}\ell(u) \right.\\
								&&&\left. -\lambda(Q_i)\sum_{e\in \Gamma(Q_i)}w(e)-\lambda(Q_i)\sum_{u\in Q_i}\ell(u)\right) \\
								&< && \tilde{x}_{Q_i} \left(\lambda( Q_i \cap Q_{i+1})\sum_{e\in\Gamma(Q_i \cap Q_{i+1})}w(e) + \lambda(Q_i \setminus Q_{i+1})\sum_{e\in\Gamma(Q_i \setminus Q_{i+1})}w(e) \right.\\
								&&&\left.-\lambda(Q_i)\sum_{e\in \Gamma(Q_i)}w(e)\right)\\
								&\leq && \tilde{x}_{Q_i} \left(\left[\sum_{e\in\Gamma(Q_i \cap Q_{i+1})}w(e)-\sum_{e\in \Gamma(Q_i)}w(e)\right]\lambda(Q_i \cap Q_{i+1})\right. \\
								&&&\left.+ \left[\sum_{e\in\Gamma(Q_i \setminus Q_{i+1})}w(e)-\sum_{e\in \Gamma(Q_i)}w(e)\right]\lambda(Q_i \setminus Q_{i+1}) \right),
\end{IEEEeqnarray*}
where the first inequality follows from the superadditivity of $\lambda$ and the second inequality holds as $\lambda(Q_i \cap Q_{i+1})\sum_{u\in Q_i\cap Q_{i+1}}\ell(u) +\lambda(Q_i \setminus Q_{i+1})\sum_{u\in Q_i\setminus Q_{i+1}}\ell(u) < \lambda(Q_i)\sum_{u\in Q_i}\ell(u)$, because $Q_i\cap Q_{i+1}$ and $Q_i\setminus Q_{i+1}$ partition $Q_i$, $\lambda(Q)>0$ for all $Q\subseteq V$, and $\ell(v)>0$ for all $v\in V$. The last inequality follows after rearranging terms and using the superadditivity of the failure rate function $\lambda$. We continue by proving that the right hand side of the last equality is less than zero; i.e., we show that $\sum_{e\in\Gamma(Q_i\cap Q_{i+1})}w(e) \leq \sum_{e\in\Gamma(Q_i)}w(e)$ and $\sum_{e\in \Gamma(Q_i\setminus Q_{i+1})}w(e) \leq \sum_{e\in \Gamma(Q_i)}w(e)$. We have 
{\allowdisplaybreaks
\begin{IEEEeqnarray*}{rCll}
\lefteqn{\pi(S^{\prime})-\pi(\tilde{S})} \label{eq:suboptimal_LRU_cycle}\\
								&< && \tilde{x}_{Q_i} \left(\left[\sum_{e\in\Gamma(Q_i \cap Q_{i+1})}w(e)-\sum_{e\in \Gamma(Q_i)}w(e)\right]\lambda(Q_i \cap Q_{i+1}) \right.\\
								&&&\left. + \left[\sum_{e\in\Gamma(Q_i \setminus Q_{i+1})}w(e)-\sum_{e\in \Gamma(Q_i)}w(e)\right]\lambda(Q_i \setminus Q_{i+1}) \right)\\
								&= && \tilde{x}_{Q_i} \left(\left[\sum_{e\in\Gamma(Q_i \cap Q_{i+1})}w(e)-\sum_{e\in \mathcal{F}(Q_i\cap Q_{i-1},Q_{i-1})}w(e)-\sum_{e\in \Gamma(Q_i)\setminus \mathcal{F}(Q_i\cap Q_{i-1},Q_{i-1})}w(e)\right]\lambda(Q_i \cap Q_{i+1})\right.\\
								&&&\left. + \left[\sum_{e\in\Gamma(Q_i \setminus Q_{i+1})}w(e)-\sum_{e\in \mathcal{F}(Q_i\cap Q_{i+1},Q_{i+1})}w(e)-\sum_{e\in \Gamma(Q_i)\setminus \mathcal{F}(Q_i\cap Q_{i+1},Q_{i+1})}w(e)\right]\lambda(Q_i \setminus Q_{i+1})\right)\\
								&\leq && \tilde{x}_{Q_i} \left(\left[\sum_{e\in\Gamma(Q_i \cap Q_{i+1})}w(e)-\sum_{e\in \mathcal{F}(Q_i\cap Q_{i-1},Q_{i-1})}w(e)-\sum_{e\in \Gamma(Q_i\cap Q_{i+1})\setminus \mathcal{F}(Q_i\cap Q_{i+1},Q_i)}w(e)\right]\lambda(Q_i \cap Q_{i+1})\right.\\
								&&&\left. + \left[\sum_{e\in\Gamma(Q_i \setminus Q_{i+1})}w(e)-\sum_{e\in \mathcal{F}(Q_i\cap Q_{i+1},Q_{i+1})}w(e)-\sum_{e\in \Gamma(Q_i\setminus Q_{i+1})\setminus \mathcal{F}(Q_i\cap Q_{i+1},Q_i)}w(e)\right]\lambda(Q_i \setminus Q_{i+1}) \right)\\
								&\leq && \tilde{x}_{Q_i} \left(\left[\sum_{e\in\Gamma(Q_i \cap Q_{i+1})}w(e)-\sum_{e\in \mathcal{F}(Q_i\cap Q_{i+1},Q_i)}w(e)-\sum_{e\in \Gamma(Q_i\cap Q_{i+1})\setminus \mathcal{F}(Q_i\cap Q_{i+1},Q_i)}w(e)\right]\lambda(Q_i \cap Q_{i+1})\right.\\
								&&&\left. + \left[\sum_{e\in\Gamma(Q_i \setminus Q_{i+1})}w(e)-\sum_{e\in \mathcal{F}(Q_i\cap Q_{i+1},Q_i)}w(e)-\sum_{e\in \Gamma(Q_i\setminus Q_{i+1})\setminus \mathcal{F}(Q_i\cap Q_{i+1},Q_i)}w(e)\right]\lambda(Q_i \setminus Q_{i+1})\right)=0.
\end{IEEEeqnarray*}
}
The equality holds, because $\mathcal{F}(Q_i\cap Q_{i-1},Q_{i-1})\subset \Gamma(Q_i)$ and $\mathcal{F}(Q_i\cap Q_{i+1},Q_{i+1})\subset \Gamma(Q_i)$. The second inequality holds by Lemma \ref{lem:removal_path}. The last inequality follows because $Q_i=\argmin_{Q_j\in C}\{W_j\}$ and we assumed that $W_i=\sum_{e\in\mathcal{F}(Q_i\cap Q_{i+1},Q_i)}w(e)$. 
Hence, $x^{\prime}$ is a solution without the LRU cycle $C$ and satisfies $\pi(S^{\prime})< \pi(\tilde{S})$.

Next, consider the case $W_i=\sum_{e\in \mathcal{F}(Q_i\cap Q_{i-1},Q_i)}w(e)$. Then, we create the solution $x^{\prime}$ by partitioning $Q_i$ in $Q_i\cap Q_{i-1}$ and $Q_i\setminus Q_{i-1}$, and we follow the same procedure as above using Lemma \ref{lem:removal_path} and $W_i=\sum_{e\in \mathcal{F}(Q_i\cap Q_{i-1},Q_i)}w(e)$. Hence, the solution $x^{\prime}$ does not have the LRU cycle $C$ and satisfies $\pi(S^{\prime})< \pi(\tilde{S})$. 
%
%
\endproof

The result of Theorem \ref{thm:intersection_cycle} is illustrated by means of Example \ref{ex:illustrate_LRU_cycle}.

\begin{example}
Suppose we have the connection graph $G$ with $w(e)=1,\;\forall e\in E$ and the precedence graph $D$ from Figure \ref{fig:illustrate_lemma3}. Furthermore, we consider the solution $\tilde{x}$ with a LRU cycle $C\subseteq \tilde{S}$ as drawn by the dashed ellipses in Figure \ref{fig:illustrate_lemma3}. We illustrate our procedure for splitting LRU a $Q_i$ into $Q_i\cap Q_{i+1}$ and $Q_i\setminus Q_{i+1}$. 
\begin{figure}[htbp!]
\centering
\subcaptionbox{Connection Graph}{
\begin{tikzpicture}[scale=1,transform shape,rotate=90]
\tikzset{LabelStyle/.style = {draw=none,fill=none,text=black}}
\tikzset{EdgeStyle/.style   = {draw,fill=none}}

	\tikzset{VertexStyle/.style = {circle,draw=none,fill=white,rotate=-90}}
	\node[VertexStyle] (100) at (0,-1.85) {$Q_i$};
	\node[VertexStyle] (101) at (-4.85,3) {$Q_{i+1}$};
	\node[VertexStyle] (102) at (4.85,3) {$Q_{i-1}$};
	
	\tikzset{VertexStyle/.style = {circle,draw,fill=white,rotate=-90}}	
	\node[VertexStyle] (1) at (0,-1) {1};
	\node[VertexStyle] (2) at (-3,0) {2};
	\node[VertexStyle] (3) at (3,0) {3};
	\node[VertexStyle] (4) at (0,1) {4};
	
	\node[VertexStyle] (5) at (2,3) {5};
	\node[VertexStyle] (6) at (4,3) {6};
	\node[VertexStyle] (7) at (3,6) {7};
	
	\node[VertexStyle] (8) at (2,8) {8};
	\node[VertexStyle] (9) at (0,10) {9};
	\node[VertexStyle] (10) at (-2,8) {10};
	
	\node[VertexStyle] (11) at (-3,6) {11};
	\node[VertexStyle] (12) at (-2,3) {12};
	\node[VertexStyle] (13) at (-4,3) {13};
	
	\tikzset{EdgeStyle/.style   = {draw,fill=none}}
	\draw[EdgeStyle] (2) to node[LabelStyle]{} (1);
	\draw[EdgeStyle] (2) to node[LabelStyle]{} (4);
	\draw[EdgeStyle] (4) to node[LabelStyle]{} (3);
	\draw[EdgeStyle] (4) to node[LabelStyle]{} (12);
	\draw[EdgeStyle] (4) to node[LabelStyle]{} (5);
	\draw[EdgeStyle] (1) to node[LabelStyle]{} (3);
	\draw[EdgeStyle] (3) to node[LabelStyle]{} (6);
	\draw[EdgeStyle] (3) to node[LabelStyle]{} (5);
	\draw[EdgeStyle] (5) to node[LabelStyle]{} (7);
	\draw[EdgeStyle] (6) to node[LabelStyle]{} (7);
	\draw[EdgeStyle] (7) to node[LabelStyle]{} (8);
	\draw[EdgeStyle] (8) to node[LabelStyle]{} (9);
	\draw[EdgeStyle] (9) to node[LabelStyle]{} (10);
	\draw[EdgeStyle] (10) to node[LabelStyle]{} (11);
	\draw[EdgeStyle] (11) to node[LabelStyle]{} (12);
	\draw[EdgeStyle] (11) to node[LabelStyle]{} (13);
	\draw[EdgeStyle] (12) to node[LabelStyle]{} (2);
	\draw[EdgeStyle] (13) to node[LabelStyle]{} (2);
	\draw[EdgeStyle] (11) to node[LabelStyle]{} (5);

 	\draw[densely dashed] (0,0) ellipse (3.5 and 1.5);
	\draw[densely dashed] (3,3) ellipse (1.5 and 3.5);
	\draw[densely dashed] (-3,3) ellipse (1.5 and 3.5);
	\draw[densely dashed,rotate=33] (5.5,5.75) ellipse (1.5 and 3.5);
	\draw[densely dashed,rotate=-33] (-5.5,5.75) ellipse (1.5 and 3.5);

\end{tikzpicture}}

\vspace{1.5cm}

\subcaptionbox{Precedence Graph}{
\begin{tikzpicture}[scale=1,transform shape]
\tikzset{VertexStyle/.style = {circle,draw,fill=white,minimum size=1.75cm}}
\tikzset{LabelStyle/.style = {draw=none,fill=none}}
\tikzset{EdgeStyle/.style   = {->,>=stealth,draw,fill=none}}
	\node[VertexStyle] (1) at (0,0) {\{{5,11}\}};
	\node[VertexStyle] (2) at (2.5,0) {\{{11,12}\}};
	\node[VertexStyle] (3) at (5,0) {\{{12,2}\}};
	\node[VertexStyle] (4) at (7.5,0) {\{{2,4}\}};
	\node[VertexStyle] (5) at (10,0) {\{{4,3}\}};
	\node[VertexStyle] (6) at (12.5,0) {\{{3,5}\}};
	
	\draw[EdgeStyle] (1) to node[LabelStyle]{} (2);
	\draw[EdgeStyle] (2) to node[LabelStyle]{} (3);
	\draw[EdgeStyle] (3) to node[LabelStyle]{} (4);
	\draw[EdgeStyle] (4) to node[LabelStyle]{} (5);
	\draw[EdgeStyle] (5) to node[LabelStyle]{} (6);
\end{tikzpicture}}
\caption{The input graphs and a LRU cycle $C$}
\label{fig:illustrate_lemma3}
\end{figure}
In this example, we have 
{\allowdisplaybreaks
\begin{IEEEeqnarray*}{rCl}
\Gamma(Q_i)									&=&\{\{2,12\},\{2,13\},\{2,4\},\{3,4\},\{3,5\},\{3,6\},\{4,5\},\{4,12\}\},\\
\Gamma(Q_{i+1})								&=&\{\{1,2\},\{10,11\},\{4,12\},\{5,11\},\{11,12\},\{12,2\},\{2,4\},\{4,3\},\{3,5\}\},\\
\Gamma(Q_{i-1})								&=&\{\{1,3\},\{4,5\},\{7,8\},\{5,11\},\{11,12\},\{2,12\},\{2,4\},\{3,4\},\{3,5\}\},\\
\Gamma(Q_i\cap Q_{i-1})						&=&\{\{1,3\},\{4,3\},\{3,5\},\{3,6\}\},\\
\Gamma(Q_i\setminus Q_{i-1})				&=&\{\{1,3\},\{2,12\},\{2,13\},\{2,4\},\{3,4\},\{3,5\},\{4,5\},\{4,12\}\},\\
\mathcal{F}(Q_i\cap Q_{i+1},Q_{i+1})		&=&\{\{2,12\},\{2,13\}\},\\
\mathcal{F}(Q_i\cap Q_{i-1},Q_i)			&=&\{\{1,3\}\},\\
\mathcal{F}(Q_i\cap Q_{i-1},Q_{i-1})		&=&\{\{3,6\}\}. 
\end{IEEEeqnarray*}
}
One can use the above expressions to verify that the procedure in the proof of Theorem \ref{thm:intersection_cycle} yields a solution $x^{\prime}$ such that $\pi(S^{\prime})< \pi(\tilde{S})$. \hfill $\diamond$
\label{ex:illustrate_LRU_cycle}
\end{example}

%
%
Next, we introduce the concept of totally balanced matrices and Theorem \ref{thm:totally_balanced} that -- in combination with Theorem \ref{thm:intersection_cycle} -- helps us to prove that an optimal solution to LPM is integer. The definition of a totally balanced matrix can for instance be found in \citet{anstee1984characterizations} or \citet{hoffman1985totally}.

\begin{definition}
A binary matrix $\mathcal{R}$ is totally balanced if it does not contain a square submatrix $R$ that has no identical columns and the sum of each row and column equals to two. 
\label{def:totally_balanced}
\end{definition}

Given this definition, we can now introduce Theorem \ref{thm:totally_balanced}, stating that the matrix encoding of an optimal solution to LPM is totally balanced. Note that we do not consider the constraint matrix of LPM -- as is mostly done -- but we study the matrix encoding of an optimal solution to LPM.

\begin{theorem}
If an optimal solution $x^*$ does not contain a LRU cycle, then $\mathcal{Z}^*$ is totally balanced. 
\label{thm:totally_balanced}
\end{theorem}
\proof{}
By the definition of a totally balanced matrix the statement of this theorem is equivalent to: Given an optimal solution $x^*$ -- with $S^*$ or equivalently $\mathcal{Z}^*$ -- that does not contain a LRU cycle, there does not exist a binary $k\times k$ submatrix $R$ of $\mathcal{Z}^*$ with $k\geq 3$, no identical columns, and such that the sum of each row and column of $R$ equals to two. We prove this by the contraposition, i.e. we prove that if such a submatrix exists, the solution contains a LRU cycle.

To be more precise, we prove the following statement: 
Given a binary $k\times k$ matrix $R$ with $k\geq 3$, no identical columns, and such that each row and column sum to two, there exists an $n\times n$ submatrix $\hat{R}$ of $R$ with minimal $n\geq 3$, no identical columns, and such that each row and column of $\hat{R}$ sum to two. If we interpret the columns of $\hat{R}$ as LRUs and the rows of $\hat{R}$ as vertices, then the LRUs -- corresponding to the columns of $\hat{R}$ -- are a LRU cycle.

Consider an $n\times n$ submatrix $\hat{R}$ of $R$ with minimal $n\geq 3$, no identical columns, and such that each row and column of $\hat{R}$ sum to two. Such a submatrix $\hat{R}$ exists, because $R$ satisfies the same conditions. We will rename the rows and columns of $\hat{R}$ such that we can easily show that the columns (LRUs) of $\hat{R}$ are a LRU cycle. 
This renaming procedure is as follows.

The first row is $v_1$ and $Q_1$ and $Q_2$ are the columns such that $\hat{r}_{v_1,Q_1}=\hat{r}_{v_1,Q_2}=1$. This follows without loss of generality, because $\hat{R}$ is binary and the sum of each row equals two. Furthermore, note that all other values of $v_1$ are zero. Next, let $v_2$ be the second row such that $\hat{r}_{v_2,Q_2}=1$. This is feasible because the sum of column $Q_2$ is two. Moreover, all other rows (except $v_1$ and $v_2$) have the value 0 in column $Q_2$. We also remark that $\hat{r}_{v_2,Q_1}=0$, since otherwise all other values in column $Q_1$ (except for $\hat{r}_{v_1,Q_1}$ and $\hat{r}_{v_2,Q_1}$) are zero and this means that columns $Q_1$ and $Q_2$ are identical, which is a contradiction. 

Next, we label the column $Q_i$ such that $\hat{r}_{v_{i-1},Q_i}=1$ for each $i=3,\ldots,n$, and we call the row $v_i$ that satisfies $\hat{r}_{v_i,Q_i}=1$, for all $i=3,\ldots,n$. This can be done due to the following reasoning. The columns $Q_j$ with $1<j<i-1$ are such that $\hat{r}_{v_{i-1},Q_j}=0$, because each column $Q_j$ already sums to two. Unless $i=n$, we have $\hat{r}_{v_{i-1},Q_1}=0$ because otherwise we would have a $i\times i$ submatrix for which each row and column sum equal 2, $i\geq 3$, and where no identical columns exists. But this would contradict the fact that $n$ is minimal. Hence, we can label $Q_i$ such that $\hat{r}_{v_{i-1},Q_i}=1$. 
Moreover, all rows $v_j$ with $1\leq j<i-1$ are such that $\hat{r}_{v_j,Q_i}=0$, because each row $v_j$ already sums to two (by considering columns $Q_k$ with $k< i$). Therefore, we can call a row $v_i$ such that $\hat{r}_{v_i,Q_i}=1$. 


Finally, we let $r_{v_n,Q_1}=1$ such that the row and column sum of \emph{each} row and column of $\hat{R}$ equals 2.

Given the renaming of the columns and rows of $\hat{R}$, we have for all $1\leq i\leq n$ that $\{v_i\}=Q_i\cap Q_{i+1}$ with $n+1 \equiv 1 \pmod{n}$. Furthermore, this implies that $(Q_i\cap Q_{i+1})\setminus (Q_{i+1}\cap Q_{i+2})=\{v_i\} \setminus \{v_{i+1}\}\neq\emptyset$ and $(Q_{i+1}\cap Q_{i+2})\setminus (Q_i\cap Q_{i+1})=\{v_{i+1}\} \setminus \{v_{i}\}\neq\emptyset$ with $n+1 \equiv 1 \pmod{n}$ and $n+2 \equiv 2 \pmod{n}$. This implies that the LRUs $Q_1,Q_2,\ldots, Q_n$ are a LRU cycle. 

Hence, we proved that the submatrix $\hat{R}$ corresponds to a LRU cycle. Hence, if an optimal solution $x^*$ does not contain a LRU cycle, it will not contain such a submatrix and, thus, the matrix $\mathcal{Z}^*$ is totally balanced. 
\endproof

Given Theorem \ref{thm:intersection_cycle} and Theorem \ref{thm:totally_balanced}, our final and rather unusual result follows relatively easily in Theorem \ref{thm:integrality}.

\begin{theorem}
All extreme optimal points of the feasible polyhedron of LPM are integer.
\label{thm:integrality}
\end{theorem}
\proof{Proof.}
Let $x^*$ be an optimal solution to LPM. Each LRU $Q\in S^*$ is a connected subgraph of $G$ by Lemma \ref{lem:connected_LRU} and $x^*$ does not contain a LRU cycle by Theorem \ref{thm:intersection_cycle}. Then, the matrix $\mathcal{Z}^*$ is totally balanced by Theorem \ref{thm:totally_balanced}. Consequently, the polyhedron $\mathcal{P}=\left\{x: \mathcal{Z}^*x=1,x\geq 0,x\in\mathbb{R}^{|S^*|}\right\}$ is integral \citep{fulkerson1974balanced}. Hence, $x^*$ is either integer or a convex combination of integer solutions to LPM, and thus we obtain our desired result. 
\endproof

In a column generation approach, a restricted version of LPM is solved repeatedly after new columns are added (this will be detailed in Section \ref{sec:solving_LPM}). 
If the restricted versions of LPM are solved by the simplex algorithm, we obtain an optimal solution $x^*$ that is an extreme point of the polyhedron $P=\left\{x : Zx=1,x\geq 0,x\in\mathbb{R}^{|\mathscr{S}|}\right\}$, but is also an extreme point of the polyhedron $\mathcal{P}=\left\{x: \mathcal{Z}^* x=1,x\geq 0,x\in \mathbb{R}^{|S^*|}\right\}$ spanned by the submatrix $\mathcal{Z}^*$. Theorem \ref{thm:integrality} now implies that $x^*$ is integral, because $x^*$ is an extreme point of $\mathcal{P}$. Hence, solving LPM with the simplex algorithm yields an optimal integer solution, and this solution is thus also optimal for M.

If LPM is solved by the simplex algorithm, we obtain an optimal solution $x^*$ that is an extreme point of the polyhedron $P=\left\{x : Zx=1,x\geq 0,x\in\mathbb{R}^{|\mathscr{S}|}\right\}$, but is also an extreme point of the polyhedron $\mathcal{P}=\left\{x: \mathcal{Z}^* x=1,x\geq 0,x\in \mathbb{R}^{|S^*|}\right\}$ spanned by the submatrix $\mathcal{Z}^*$. Theorem \ref{thm:integrality} now implies that $x^*$ is integral, because $x^*$ is an extreme point of $\mathcal{P}$. Hence, solving LPM with the simplex algorithm yields an optimal integer solution, and this solution is thus also optimal for M.

We would like to stress that our result in Theorem \ref{thm:integrality} is rather unusual, because the polyhedron $P$ of LPM is \emph{not} integral. This contrasts with much other research that focuses on proving integrality of the polyhedron to conclude that an optimal integer solution can be found (if the objective function is convex), e.g. via totally unimodular constraint matrices. 
We demonstrate that -- for non--integral polyhedra -- analysis of the objective function can be used to establish the existence of an optimal integer solution to a relaxed problem when the constraint matrix will not guarantee an integral polyhedron. The crucial analysis of the objective function is in the proof of Theorem \ref{thm:intersection_cycle}.
We believe that our proof approach can be used more generally for partitioning problems where the objective function is convex. One is only required to show that cycles in solutions are suboptimal. 

%
%
%
%
%
%

\subsection{Solving LPM and M}
\label{sec:solving_LPM}
Given Theorem \ref{thm:integrality}, we move our attention to solving LPM, for which we apply column generation. Hence, we consider a feasible subset of LRUs (or columns) $\tilde{\mathscr{S}}\subseteq\mathscr{S}$ for LPM. This results in the \textit{Restricted Master Program} (RLPM). 
%
%
%
%
%
%
%
%
%
%
For RLPM, we generate profitable LRUs (columns) by solving the pricing problem of RLPM:
\begin{equation}
c^*=\min_{Q\in \mathscr{S}}\left\{\omega(Q)-\sum_{v\in Q}r_v\right\},
\label{eq:pricing_compact}
\end{equation}
where $r_v$ are dual variables for the partitioning constraints of RLPM. 
We want to find a LRU $Q\in\mathscr{S}$ with minimal reduced cost. 
After we solve the pricing problem \eqref{eq:pricing_lin}, we add the obtained LRU to $\tilde{\mathscr{S}}$ and we again solve LPM with the new $\tilde{\mathscr{S}}$. Next, we solve the pricing problem again, and we repeat this procedure until the pricing problem does not return a profitable LRU (column), i.e., we terminate when $c^*\geq 0$. This means that there does not exist a LRU (column) that is worthwhile to add to our LPM, and we have obtained the optimal solution. 

The computation of the pricing problem is difficult in general, but can be done with a standard solver when the failure rate function $\lambda$ is additive, i.e., when $\lambda(S)=\sum_{i\in S}\lambda(\{i\})$. We provide this formulation below.

\subsubsection{Pricing problem for additive failure rates}
\label{sec:pricingaddtive}

Problem \eqref{eq:pricing_compact} can be expressed as follows when $\lambda$ is additive ($\lambda(S)=\sum_{i\in S}\lambda(\{i\})$). The binary auxiliary variable $k^e$ denotes whether edge $e\in E$ needs to be disconnected to remove LRU $Q$. The binary decision variable $\gamma_v$ indicates whether node $v$ is included in the LRU $Q$. Now we rewrite Problem \eqref{eq:pricing_compact} to obtain
\begin{IEEEeqnarray}{lllllll}
\interdisplaylinepenalty=0
\IEEEyesnumber\label{eq:pricing_nonlin} \IEEEyessubnumber*
\mbox{$c^*$=}	\qquad		&& \underset{\boldsymbol{\gamma},\boldsymbol{k}}{\min}	&\quad	& \sum_{e\in E}k^e w(e) \sum_{v\in V}\gamma_{v}\lambda(\{v\})+\sum_{u\in V}&&\gamma_{u}\ell(u)\sum_{v\in V}\gamma_{v}\lambda(\{v\})-\sum_{v\in V}\gamma_v r_v \label{eq:pricing_nonlin_obj}\IEEEeqnarraynumspace\\
							&& \mbox{s.t.}		&		& \gamma_{u}-\gamma_{v}\leq k^e,	& &\forall \{u,v\}\in E, \forall e\in H(\{u,v\}), \\
							&&					&		& \gamma_{v},k^e \in \{0,1\}.										&&
\end{IEEEeqnarray}
Problem \eqref{eq:pricing_compact} can be linearized by applying the McCormick linearization method \citep{mccormick1976computability}. In particular, let
 $\eta_{ev}=k^e\gamma_v$ and $\delta_{uv}=\gamma_u\gamma_v$. Note that $\eta_{ev}$ denotes whether the LRU contains part $v\in V$ and that edge $e\in E$ needs to be broken for the LRU to be removed. Similarly, $\delta_{uv}$ represents whether the LRU contains both parts $u,v\in V$. Then a binary linear formulation of \eqref{eq:pricing_compact} is given below in\eqref{eq:pricing_lin}. 

\begin{IEEEeqnarray}{lllllll}
\interdisplaylinepenalty=0
\IEEEyesnumber\label{eq:pricing_lin} \IEEEyessubnumber*
\mbox{$c^*$=}	\qquad		&& \underset{\boldsymbol{k},\boldsymbol{\gamma},\boldsymbol{\eta},\boldsymbol{\delta}}{\min}	&\quad	& \sum_{e\in E} \sum_{v\in V}\eta_{ev}\lambda(v)w(e)+\sum_{u,v\in V}&&\delta_{uv}\ell(u)\lambda(v)-\sum_{v\in V}\gamma_v r_v\\
							&& \mbox{s.t.}		&		& \gamma_{u}-\gamma_{v}\leq k^e,								&&\forall \{u,v\}\in E, \forall e\in H(\{u,v\}),\IEEEeqnarraynumspace\\
							&&					&		& \eta_{ev}\leq k^e,											&&\forall e\in E, \forall v\in V,\\
							&&					&		& \eta_{ev}\leq \gamma_v,										&&\forall e\in E, \forall v\in V,\\
							&&					&		& \eta_{ev}\geq k^e+\gamma_v-1 ,								&&\forall e\in E, \forall v\in V,\\
							&&					&		& \delta_{uv}\leq \gamma_u,										&&\forall u,v\in V,\\
							&&					&		& \delta_{uv}\leq \gamma_v,										&&\forall u,v\in V,\\
							&&					&		& \delta_{uv}\geq \gamma_u+\gamma_v-1,							&&\forall u,v\in V,\\
							&&					&		& \eta_{ev},\delta_{uv}\geq 0,\gamma_{v},k^e \in \{0,1\}.			&&
\end{IEEEeqnarray}

After we solve the pricing problem \eqref{eq:pricing_lin}, we add the obtained LRU to $\tilde{\mathscr{S}}$ and we again solve LPM with the new $\tilde{\mathscr{S}}$. Next, we solve the pricing problem again, and we repeat this procedure until the pricing problem does not return a profitable LRU (column), i.e., we terminate when $c^*\geq 0$. This means that there does not exist a LRU (column) that is worthwhile to add to our LPM, and we have obtained the optimal solution.

\section{Binary programming formulation an additive failure rate function}
\label{sec:bp}

In this section, we consider only additive failure rates $\lambda$, i.e. $\lambda(Q)=\sum_{v\in Q}\lambda(\{v\})$. This special case is convenient when implementing the model and solving it with a standard solver like CPLEX or Gurobi. We formulate \textsc{LRU Design} as a binary non--linear program (BNLP) and linearize it so that it can be solved by a standard integer program solver. We use this linearlized binary linear Program (BLP) as a benchmark to compare our set partitioning formulation with. 

For the BNLP, we first relax the fact that $\emptyset\not\in S$. In the foregoing, $S$ was a partition of $V$. For the BNLP (and later the BLP) we consider a solution $S^{\prime}$ satisfying $|S^{\prime}|=|V|$, where $S^{\prime}$ may contain empty LRUs, and we have that $S=\{Q\in S^{\prime}: Q\neq \emptyset\}$. Next, we index each LRU in $S^{\prime}$ by $i\in\{1,\ldots,|V|\}$, i.e., we have LRUs $Q_i\in S^{\prime}$ that are indexed by $i$. Furthermore, we create a binary variable $y_{vi}$ that indicates whether a part $v\in V$ is assigned to LRU $Q_i,\;i\in\{1,\ldots,|V|\}$: 
$$y_{vi}=\begin{cases} 1 &\quad \text{if $v\in Q_i$} \\ 0 &\quad \text{otherwise}\end{cases},\quad \forall v\in V, \forall i\in \{1,\ldots,|V|\}.$$
We denote $\boldsymbol{Y}$ as the matrix consisting of all entries $y_{vi}$. Note that we can derive $S^{\prime}$ easily from $\boldsymbol{Y}$. 
We also define the auxiliary binary variable $k_i^e$ that denotes whether edge $e\in E$ needs to be broken in order to replace LRU $Q_i$, and organize them in a matrix $\boldsymbol{K}$. We determine the value of $k^e_i$ by considering all the edges $b\in B(Q_i)$. We determine $B(Q_i)$ by considering the edges $\{u,v\}\in E$ such that $y_{ui}-y_{vi}=1$ or $y_{vi}-y_{ui}=1$, i.e., one of the end points of $\{u,v\}$ belongs to LRU $Q_i$ while the other end point does not. This corresponds to the definition of $B(Q_i)$. Subsequently, we consider each edge $e\in E$ that needs to be broken before breaking $\{u,v\}$; i.e., for each $\{u,v\}\in B(Q_i)$ we consider all $e\in H(\{u,v\})$. Hence, the variable $k_i^e$ satisfies $y_{ui}-y_{vi}\leq k_i^e,\; \forall \{u,v\}\in E, \forall e\in H(\{u,v\}),\forall i\in \{1,\ldots,|V|\}$. 
Note that $k_i^e$ may take the value of one, even if an edge $e\in E$ is fully contained within a LRU $Q_i$.
We denote $\boldsymbol{K}$ as the matrix consisting of all entries $k_i^e$.
We use the variable $k_i^e$ in our objective function \eqref{eq:bnlp_obj}, since it represents whether an edge has to be broken ($k^e_i=1$) in order to remove LRU $Q_i$. Furthermore, the objective function implies $k^e_i=0$ if edge $e$ is not broken for the replacement of $Q_i$. 
Next, we use $k_i^e$, the edge weights, and the failure rate of $Q_i$ (expressed using $y_{vi}$) to determine the cost for replacing $Q_i$. The total purchase cost of the LRU $Q_i$ is derived by using $y_{vi}$, and we multiply this by the total failure rate of $Q_i$ (which also depends on $y_{vi}$).
This results in a binary non--linear programming formulation of \textsc{LRU Design}:
\begin{IEEEeqnarray}{lllllll}
\interdisplaylinepenalty=0
\IEEEyesnumber \label{eq:bnlp} \IEEEyessubnumber*
\mbox{(BNLP)}\qquad		&& \underset{\boldsymbol{Y},\boldsymbol{K}}{\min}	&\;	& \sum_{i=1}^{|V|} \sum_{e\in E}k_i^e w(e)\sum_{v\in V}y_{vi}\lambda(v)+&&\sum_{i=1}^{|V|}\sum_{u\in V}y_{ui}\ell(u)\sum_{v\in V}y_{vi}\lambda(v)\label{eq:bnlp_obj}\\
			&& \mbox{s.t.}		&		& \sum_{i=1}^{|V|} y_{vi}= 1,		&&\forall v\in V, \label{eq:bnlp_con_part}\\
			&& 					&		& y_{ui}-y_{vi}\leq k_i^e,			&&\forall \{u,v\} \in E,\forall e\in H(\{u,v\}),\forall i\in \{1,\ldots,|V|\}, \label{eq:bnlp_con_edgeCount}\\
			&&					&		& y_{vi},k_i^e \in \{0,1\}.			&&
\end{IEEEeqnarray}

Constraints \eqref{eq:bnlp_con_part} ensure that each part $v\in V$ is included in exactly one LRU, and constraints \eqref{eq:bnlp_con_edgeCount} enforce the definition of the auxiliary variable $k_i^e$.
%

The BNLP is a problem with a quadratic objective function. Therefore, the BNLP van be linearized by the McCormick reformulation \citep{mccormick1976computability}; i.e., we introduce new variables $\rho_{ev}^i=y_{vi}k_i^e$ and $\sigma_{uv}^i=y_{ui}y_{vi}$. The variable $\rho_{ev}^i$ denotes whether LRU $Q_i$ contains part $v\in V$ and whether edge $e\in E$ needs to be broken in order to replace the LRU $Q_i$. Analogously, $\sigma_{uv}^i$ denotes whether two parts $u,v\in V$ are both contained in the same LRU $Q_i$. Substituting $\rho_{ev}^{i}$ and $\sigma_{uv}^i$ into the above implies that we need to add constraints that enforce the interpretation we gave. Furthermore, we optimize over $\boldsymbol{Y}$, $\boldsymbol{K}$, $\boldsymbol{\rho}$ and $\boldsymbol{\sigma}$, where $\boldsymbol{\rho}$ and $\boldsymbol{\sigma}$ correspond to the 3--D arrays with entries $\rho_{ev}^{i}$ and $\sigma_{uv}^i$, respectively. Hence, we obtain the BLP formulation of \textsc{LRU Design}:
\begin{IEEEeqnarray}{lllllll}
\interdisplaylinepenalty=0
\IEEEyesnumber \label{eq:blp} \IEEEyessubnumber*
\mbox{(BLP)}\qquad			&& \underset{\boldsymbol{Y},\boldsymbol{K},\boldsymbol{\rho},\boldsymbol{\sigma}}{\min}	&\;	& \sum_{i=1}^{|V|} \sum_{e\in E}\sum_{v\in V}\rho_{ev}^i \lambda(v)w(e)+&&\sum_{i=1}^{|V|}\sum_{u,v\in V}\sigma_{uv}^i\ell(u)\lambda(v)\\
							&& \mbox{s.t.}		&		& \sum_{i=1}^{|V|} y_{vi}= 1,						&&\forall v\in V, \\
							&& 					&		& y_{ui}-y_{vi}\leq k_i^e,				&&\forall \{u,v\} \in E, \forall e\in H(\{u,v\}), \forall i\in\{1,\ldots,|V|\},\\
							&&					&		& \rho_{ev}^{i}\leq y_{vi},						&&\forall e\in E, \forall v\in V, \forall i\in\{1,\ldots,|V|\}, \label{eq:blp_subst_rho}\\
							&&					&		& \rho_{ev}^{i}\leq k_i^e,						&&\forall e\in E, \forall v\in V, \forall i\in\{1,\ldots,|V|\},\\
							&&					&		& \rho_{ev}^{i}\geq y_{vi}+k_i^e-1,				&&\forall e\in E, \forall v\in V,\forall i\in\{1,\ldots,|V|\},\\
							&&					&		& \sigma_{uv}^i\leq y_{ui},						&&\forall u,v\in V,\forall i\in\{1,\ldots,|V|\}, \IEEEeqnarraynumspace \\
							&&					&		& \sigma_{uv}^i\leq y_{vi},						&&\forall u,v\in V,\forall i\in\{1,\ldots,|V|\},\\
							&&					&		& \sigma_{uv}^i\geq y_{ui}+y_{vi}-1,			&&\forall u,v\in V,\forall i\in\{1,\ldots,|V|\},\\
							&&					&		& \rho_{ev}^i,\sigma_{uv}^i\geq 0,				&&\label{eq:blp_subst_sigma}\\
							&&					&		& y_{vi},k_i^e \in \{0,1\}.						&&
\end{IEEEeqnarray}

We acknowledge that tighter MIP formulations than BLP will exist due to extensive studies on binary quadratic programming and boolean quadratic polytopes \citep{padberg1989,boros1992,boros1993,deza1997,rendl2010,bonami2018,junger2021,charfreitag2022,rehfeldt2023}. The formulation above is a straightforward formulation that will appeal to many practitioners of operations research due to its simplicity.

\section{Numerical experiments}
\label{sec:numerical_experiments}
In this section, we use the binary linear program (BLP) formulation and the set partitioning formulation (LPM) of \textsc{LRU Design} under additive failure rate functions to gain some insight into the size of instances that can be solved to optimality.
We shed a light on the objective value gap between both approaches, and we explore the effects of parameter perturbations on our model's outcomes. Furthermore, we numerically illustrate -- based on randomly generated instances of practical size -- that the relative cost increase from introducing a partitioning constraint is very limited, and its magnitude decreases as the number of parts grows. We have implemented all optimization model formulations in JuMP \citep{lubin2015computingjulia,dunning2017jump}, which is a mathematical optimization package of Julia \citep{balbaert2016julia}, and we solved all problems using Gurobi 7.0.1 on an Intel i5--4300U @2.50GHz processor with 16GB RAM and running Ubuntu 16.04 LTS.

In Section \ref{sec:num_exp_instance_generator}, we explain the used instance generator for our experiments. In Section \ref{sec:num_exp_results} we study the difference between the computation times (in seconds) of the binary linear programming formulation (BLP) and the set partitioning formulation (LPM). Also, we discuss the relative difference in objective value between both models. Furthermore, we shed some light on how the downtime cost per time unit affects the number of LRUs used in an optimal solution. Moreover, we study how the system's complexity affects the total annual costs by considering the number of connections between parts, and the number of predecessor--successor relationships that exist. Finally, we show that the cost effect of introducing a partitioning constraint is very limited for practically sized instances. 

\subsection{Instance generator}
\label{sec:num_exp_instance_generator}
An instance is described by the graphs $G$ and $D$. We vary the number of vertices $|V|$, the number of edges $|E|$, and the number of arcs $|A|$ in our numerical experiments. We relate the number of edges in $G$ to the number of vertices by $|E|=\delta|V|$, where $\delta$ is the average vertex degree in the graph $G$. Similarly, we relate the number of arcs in $D$ to the number of edges by $|A|=\delta_E |E|$, where $\delta_E$ is the average out degree of an edge $e\in E$. All other parameters such as $\lambda(v)>0$ and $\ell(v)>0$ for all $v\in V$, and $w(e)>0$ for all $e\in E$ are randomly generated, as well as a graph's layout in terms of the edge set $E$ and the arc set $A$. 

The graphs $G$ and $D$ are generated in the following way. For $G$, we have a set of vertices $V$ and a number of unique edges $|E|=\delta|V|$, and we create a spanning tree with $|V|-1$ edges. We add an arbitrary vertex $v\in V$ to a set of considered vertices $\tilde{V}$, and we select a new vertex $u\in V\setminus \tilde{V}$ and connect it to an arbitrary vertex $z\in \tilde{V}$ by adding the edge $\{u,v\}$ to the edge set $E$. We keep doing this until $\tilde{V}=V$. Subsequently we add remaining edges randomly to our graph and we terminate once we have $|E|$ edges in $G$. 
%
%
%
%
Secondly, we generate the precedence graph $D$. We (randomly) assign an index to each edge $e\in E$ and denote this index by ${I}(e)$, and the minimum and maximum values assigned are $1$ and $|E|$, respectively. We start with $A=\emptyset$ and add an arc in each iteration. An iteration starts by selecting two random edges $\{u,v\},\{v,x\}\in E:u,v,x\in V$ and $u\neq v \neq x$. If ${I}(\{u,v\})\leq {I}(\{v,x\})$ we create an arc $(\{u,v\},\{v,x\})$ and add it to $A$, otherwise we create an arc $(\{v,x\},\{u,v\})$ and add it to $A$. We repeat this procedure until $\delta_E|E|=|A|$, and upon termination we have obtained a set $A$ that has a topological sorting and thus the precedence graph $D$ is acyclic.

\subsection{Computational results}
\label{sec:num_exp_results}

Next, we discuss the computational results for our model. 
The generation of a random graph follows the procedure from Section \ref{sec:num_exp_instance_generator}, and we let $|V|\in\{10,20,30,40,50,60\}$, $\delta\in\{2,3,4\}$, and $\delta_E\in \{0.5,1,1.5\}$. For each combination ($|V|,\delta,\delta_E)$, we generate 10 random instances, resulting in a total of 540 instances.

We use a time limit of 600 seconds for the BLP formulation and also for the set partitioning formulation. This time limit is relatively low because we solve a large number of instances, thereby making it feasible to perform the entire numerical study in a reasonable amount of time. If an instance has not been solved to optimality within 600 seconds, we say that it is inefficient. If all instances of a certain parameter combination $(|V|,\delta,\delta_E)$ are inefficient, we write -- as an entry for the combination. We determine the average computation time of both formulations based on the efficient instances. The results are presented in Table \ref{tab:comp_times_partitioning}, where the computation times are given in seconds, and the subscripts indicate the number of efficient instances. Furthermore, we have not reported computation times for the BLP with $|V|\geq 40$, since we have found no efficient solutions within the time limit. 
\begin{table}[htbp!]
\centering
\begin{tabular}{crrrr}
\hline

BLP					&							&	\multicolumn{3}{c}{$|V|$} \bigstrut[t]\\
$\delta$	&	$\delta_E$	&	10	&	20	&	30\\
\cline{3-5}
2     & 0.5   & $4.45_{10}$  & $121.02_{10}$ & --      \bigstrut[t]\\
2     & 1     & $5.16_{10}$  & $144.59_{10}$ & --      \\
2     & 1.5   & $3.39_{10}$  & $85.38_{10}$  & $577.21_{3}$     \\
3     & 0.5   & $7.18_{10}$  & $316.42_{9}$  & --  \\
3     & 1     & $6.22_{10}$  & $308.73_{10}$ & --      \\
3     & 1.5   & $4.98_{10}$  & $151.64_{10}$ & --      \\
4     & 0.5   & $11.00_{10}$ & $479.45_{6}$  & --      \\
4     & 1     & $8.56_{10}$  & $502.25_{4}$  & -- 	  \\
4     & 1.5   & $6.62_{10}$  & $246.39_{9}$  & --       \\
\hline
\end{tabular}

\bigskip

\begin{tabular}{crrrrrrr}
\hline
LPM					&							&	\multicolumn{6}{c}{$|V|$} \bigstrut[t]\\
$\delta$	&	$\delta_E$	&	10	&	20	&	30	&	40	&	50	&	60\\
\cline{3-8}
2     & 0.5   & $0.21_{10}$  & $1.64_{10}$  & $13.02_{10}$ & $37.32_{10}$  & $111.75_{10}$ & $216.53_{10}$ \bigstrut[t]\\
2     & 1     & $0.61_{10}$  & $3.79_{10}$  & $30.97_{10}$ & $88.65_{10}$  & $314.11_{9}$  & $533.70_{2}$ \\
2     & 1.5   & $0.19_{10}$  & $1.65_{10}$  & $9.66_{10}$  & $36.46_{10}$  & $124.71_{10}$ & $228.21_{9}$ \\
3     & 0.5   & $0.37_{10}$  & $2.59_{10}$  & $22.13_{10}$ & $60.02_{10}$  & $105.39_{10}$ & $331.38_{7}$ \\
3     & 1     & $0.97_{10}$  & $8.08_{10}$  & $37.98_{10}$ & $232.89_{10}$ & $331.50_{5}$  & -- \\
3     & 1.5   & $0.65_{10}$  & $3.68_{10}$  & $17.76_{10}$ & $93.86_{10}$  & $289.00_{10}$ & $514.14_{3}$ \\
4     & 0.5   & $0.50_{10}$  & $3.39_{10}$  & $26.50_{10}$ & $114.09_{10}$ & $304.42_{9}$  & $410.97_{3}$ \\
4     & 1     & $1.37_{10}$  & $10.06_{10}$ & $48.56_{10}$ & $294.81_{10}$ & $432.31_{2}$  & -- \\
4     & 1.5   & $1.29_{10}$  & $6.25_{10}$  & $32.54_{10}$ & $211.55_{10}$ & $521.40_{3}$  & -- \\
\hline
\end{tabular}
\caption{Average computation times (sec) of both formulations. The subscripts indicate that number of instances that were solved to optimality for a given setting.}
\label{tab:comp_times_partitioning}
\end{table}

We observe that, given the time limit of 600 seconds, the BLP formulation can only solve small size instances up to 20 vertices (parts), while the set partitioning (SP) formulation can solve medium size to large instances up to 60 vertices (parts).  
Furthermore, we see that the set partitioning formulation solves instances faster than the BLP formulation. This effect is amplified when the instances become larger, i.e., when $|V|$ and $\delta$ increase. 
Real--life instances are typically medium to large sized instances and can have 50 vertices (parts). Furthermore, such instances may possess many and complex connections and predecessor--successor relationships. This makes the BLP formulation unsuitable for practical purposes. Hence, it is worthwhile to invest extra time to implement the set partitioning formulation (LPM) with a pure pricing algorithm. Furthermore, the computation times illustrate that the set partitioning formulation of \textsc{LRU Design} is particularly useful as a feedback mechanism for the company's design department. The engineers can quickly assess many design alternatives (in terms of the connection graph and precedence graph) and their effects on the optimal LRU design and the corresponding (after--sales) costs.

Not solving the BLP formulation to optimality may still result in solutions that are near optimal. Therefore, we also perform a gap analysis by studying the relative difference in the objective values of both formulations. We only study this difference for parameter combinations for which each instance solved by LPM is efficient, i.e., instances with $|V|\leq 40$. We define the objective value gap $$\beta=\frac{\pi(S_b)-\pi(S_{l})}{\pi(S_l)}\times 100\%,$$ with $\pi(S_b)$ and $\pi(S_l)$ denoting the costs of the best feasible solution found after 600 seconds of the BLP and set partitioning formulation (LPM), respectively. The results are shown in Table \ref{tab:objectivevalue_gap_partitioning}.

\begin{table}[htbp!]
\centering
\begin{tabular}{crrrrr}
\hline
					&							&	\multicolumn{4}{c}{$|V|$} \bigstrut[t]\\
$\delta$	&	$\delta_E$	&	10	&	20	&	30	&	40\\
\cline{3-6}
2     & 0.5   & $0.0\%$  & $0.0\%$  & $18.9\%$ & $257.6\%$  \bigstrut[t]\\
2     & 1     & $0.0\%$  & $0.0\%$  & $20.3\%$ & $349.1\%$ \\
2     & 1.5   & $0.0\%$  & $0.0\%$  & $29.2\%$  & $369.0\%$ \\
3     & 0.5   & $0.0\%$  & $0.0\%$  & $96.7\%$ & $523.1\%$ \\
3     & 1     & $0.0\%$  & $0.1\%$  & $115.9\%$ & $520.6\%$ \\
3     & 1.5   & $0.0\%$  & $0.0\%$  & $110.3\%$ & $474.2\%$ \\
4     & 0.5   & $0.0\%$  & $0.0\%$  & $150.1\%$ & $580.2\%$ \\
4     & 1     & $0.0\%$  & $7.8\%$  & $127.5\%$ & $218.9\%$ \\
4     & 1.5   & $0.0\%$  & $4.6\%$  & $95.5\%$ & $191.1\%$ \\
\hline
\end{tabular}
\caption{Objective value gap between both formulations.}
\label{tab:objectivevalue_gap_partitioning}
\end{table}

We observe that the relative objective value difference rapidly grows as the number of vertices grows. This implies that the inefficient instances, produced by the BLP, do not result in a competitive solution compared to the set partitioning formulation (LPM).

Next, we numerically study the effect of the cost of one time unit of system downtime on the number of LRUs that is used in an optimal LRU design. In the remainder of this section, we use the same instance generator as discussed in Section \ref{sec:num_exp_instance_generator}, and we generate 1,000 instances per parameter setting $(\delta,\delta_E)$ and keep $|V|=20$. 
%
%
%
For a given instance, we vary the edge weights by multiplying all edge weights of the instance by a constant factor $q\in\{0.1, 1, 10\}$. A higher value for $q$ means that it is more expensive to break edges. If the time for breaking an edge remains constant, it means that the cost rate per time unit for breaking an edge increases, and thus we can capture a higher downtime cost per time unit by varying $q$. This way, we create three classes of instances (i) low downtime cost per time unit ($q=0.1$); (ii) moderate downtime cost per time unit ($q=1$); (iii) and high downtime cost per time unit ($q=10$). We keep the parameter values for $\delta$ and $\delta_E$ constant at $\delta=3$ and $\delta_E=1$. We focus on the number of LRUs $|S^*|$ in an optimal solution $S^*$. The results are presented in Figure \ref{fig:edge_weights}.
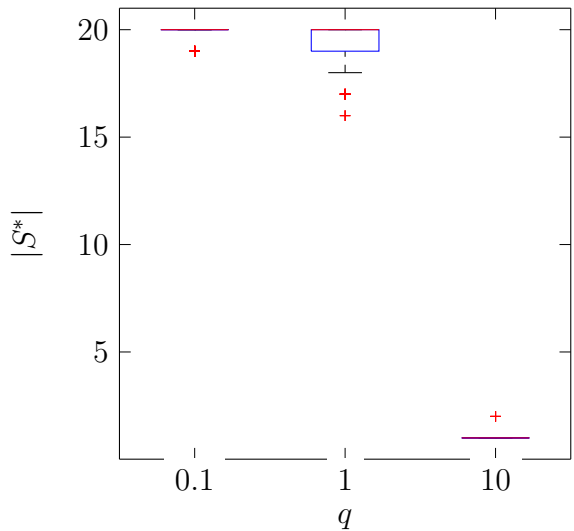
\begin{figure}[htbp!]
	\centering
	\setlength\figureheight{6cm} 
	\setlength\figurewidth{6cm}
%
%

\begin{tikzpicture}

\begin{axis}[%
width=\figurewidth,
height=\figureheight,
at={(0\figurewidth,0\figureheight)},
scale only axis,
clip=false,
separate axis lines,
every outer x axis line/.append style={black},
every x tick label/.append style={font=\color{black}},
every x tick/.append style={black},
xmin=0.5,
xmax=3.5,
xtick={1,2,3},
xlabel={$q$},
every outer y axis line/.append style={black},
every y tick label/.append style={font=\color{black}},
every y tick/.append style={black},
ymin=0.01,
ymax=21,
ylabel={$|S^*|$},
axis background/.style={fill=white}
]
\addplot [color=black, dashed, forget plot]
  table[row sep=crcr]{%
1.000	20.000\\
1.000	20.000\\
};
\addplot [color=black, dashed, forget plot]
  table[row sep=crcr]{%
2.000	20.000\\
2.000	20.000\\
};
\addplot [color=black, dashed, forget plot]
  table[row sep=crcr]{%
3.000	1.000\\
3.000	1.000\\
};
\addplot [color=black, dashed, forget plot]
  table[row sep=crcr]{%
1.000	20.000\\
1.000	20.000\\
};
\addplot [color=black, dashed, forget plot]
  table[row sep=crcr]{%
2.000	18.000\\
2.000	19.000\\
};
\addplot [color=black, dashed, forget plot]
  table[row sep=crcr]{%
3.000	1.000\\
3.000	1.000\\
};
\addplot [color=black, forget plot]
  table[row sep=crcr]{%
0.887	20.000\\
1.113	20.000\\
};
\addplot [color=black, forget plot]
  table[row sep=crcr]{%
1.888	20.000\\
2.112	20.000\\
};
\addplot [color=black, forget plot]
  table[row sep=crcr]{%
2.888	1.000\\
3.112	1.000\\
};
\addplot [color=black, forget plot]
  table[row sep=crcr]{%
0.887	20.000\\
1.113	20.000\\
};
\addplot [color=black, forget plot]
  table[row sep=crcr]{%
1.888	18.000\\
2.112	18.000\\
};
\addplot [color=black, forget plot]
  table[row sep=crcr]{%
2.888	1.000\\
3.112	1.000\\
};
\addplot [color=blue, forget plot]
  table[row sep=crcr]{%
0.775	20.000\\
0.775	20.000\\
1.225	20.000\\
1.225	20.000\\
0.775	20.000\\
};
\addplot [color=blue, forget plot]
  table[row sep=crcr]{%
1.775	19.000\\
1.775	20.000\\
2.225	20.000\\
2.225	19.000\\
1.775	19.000\\
};
\addplot [color=blue, forget plot]
  table[row sep=crcr]{%
2.775	1.000\\
2.775	1.000\\
3.225	1.000\\
3.225	1.000\\
2.775	1.000\\
};
\addplot [color=red, forget plot]
  table[row sep=crcr]{%
0.775	20.000\\
1.225	20.000\\
};
\addplot [color=red, forget plot]
  table[row sep=crcr]{%
1.775	20.000\\
2.225	20.000\\
};
\addplot [color=red, forget plot]
  table[row sep=crcr]{%
2.775	1.000\\
3.225	1.000\\
};
\addplot [color=blue, draw=none, mark=+, mark options={solid, red}, forget plot]
  table[row sep=crcr]{%
1.000	19.000\\
1.000	19.000\\
1.000	19.000\\
1.000	19.000\\
1.000	19.000\\
1.000	19.000\\
1.000	19.000\\
1.000	19.000\\
1.000	19.000\\
1.000	19.000\\
1.000	19.000\\
1.000	19.000\\
1.000	19.000\\
1.000	19.000\\
1.000	19.000\\
1.000	19.000\\
1.000	19.000\\
1.000	19.000\\
1.000	19.000\\
1.000	19.000\\
1.000	19.000\\
1.000	19.000\\
1.000	19.000\\
1.000	19.000\\
1.000	19.000\\
1.000	19.000\\
1.000	19.000\\
1.000	19.000\\
1.000	19.000\\
1.000	19.000\\
1.000	19.000\\
1.000	19.000\\
1.000	19.000\\
1.000	19.000\\
1.000	19.000\\
1.000	19.000\\
};
\addplot [color=blue, draw=none, mark=+, mark options={solid, red}, forget plot]
  table[row sep=crcr]{%
2.000	16.000\\
2.000	16.000\\
2.000	16.000\\
2.000	17.000\\
2.000	17.000\\
2.000	17.000\\
2.000	17.000\\
2.000	17.000\\
2.000	17.000\\
2.000	17.000\\
2.000	17.000\\
2.000	17.000\\
2.000	17.000\\
2.000	17.000\\
2.000	17.000\\
2.000	17.000\\
2.000	17.000\\
};
\addplot [color=blue, draw=none, mark=+, mark options={solid, red}, forget plot]
  table[row sep=crcr]{%
3.000	2.000\\
3.000	2.000\\
3.000	2.000\\
3.000	2.000\\
3.000	2.000\\
};
\node[above, align=center, fill=white]
at (1cm,-0.55cm) {0.1};
\node[above, align=center, fill=white]
at (3cm,-0.55cm) {1};
\node[above, align=center, fill=white]
at (5cm,-0.55cm) {10};
\end{axis}
\end{tikzpicture}%
	\caption{Effect of edge weights on the number of LRUs in $S^*$}
	\label{fig:edge_weights}
\end{figure}

Based on Figure \ref{fig:edge_weights}, the instances where the downtime cost per time unit is low, have many small LRUs (each part is a LRU in itself in the extreme case). These solutions prefer small LRUs because they have lower purchase costs. As the cost for a single time unit of downtime increases, we see that the optimal solution prefers fewer LRUs that each become larger, because such larger LRUs enable faster replacement and thus lower downtime costs. This explains, for example, why we observe that the consumer electronics industry with low values for $q$ has rather small LRUs. On the other end of the spectrum, capital intensive industries such as the semiconductor industry or the aviation industry have high values for $q$, and they tend to opt for larger LRUs which enable faster replacement. Both these phenomena are confirmed by the numerical results of our model.

The second effect that we study considers the complexity of the system, and how this affects the costs of the optimal LRU design. We restrict our attention, for now, on the number of edges in the connection graph $G$ that describes system complexity. We vary how strongly various parts are connected to each other by altering $\delta$. A low (high) value of $\delta$ corresponds to lesser (more) connected parts. We are interested in the effect that the number of connections in $G$ has on the costs, because this provides a justification of whether to avoid many connections between parts in order to reduce the total cost. For our analysis, we keep $\delta_E$ and $q$ constant at $\delta_E=1$ and $q=1$. The results are presented in Figure \ref{fig:effect_delta}.
\begin{figure}[htbp!]
	\centering
	\setlength\figureheight{6cm} 
	\setlength\figurewidth{6cm}
%
%
\begin{tikzpicture}

\begin{axis}[%
width=\figurewidth,
height=\figureheight,
at={(0\figurewidth,0\figureheight)},
scale only axis,
clip=false,
separate axis lines,
every outer x axis line/.append style={black},
every x tick label/.append style={font=\color{black}},
every x tick/.append style={black},
xmin=0.5,
xmax=3.5,
xtick={1,2,3},
xlabel={$\overline{\delta}$},
every outer y axis line/.append style={black},
every y tick label/.append style={font=\color{black}},
every y tick/.append style={black},
ymin=18,
ymax=175,
ylabel={$\pi(S^*)$},
axis background/.style={fill=white}
]
\addplot [color=black, dashed, forget plot]
  table[row sep=crcr]{%
1.000	64.843\\
1.000	86.616\\
};
\addplot [color=black, dashed, forget plot]
  table[row sep=crcr]{%
2.000	94.002\\
2.000	124.122\\
};
\addplot [color=black, dashed, forget plot]
  table[row sep=crcr]{%
3.000	122.383\\
3.000	158.591\\
};
\addplot [color=black, dashed, forget plot]
  table[row sep=crcr]{%
1.000	29.282\\
1.000	50.266\\
};
\addplot [color=black, dashed, forget plot]
  table[row sep=crcr]{%
2.000	43.861\\
2.000	73.682\\
};
\addplot [color=black, dashed, forget plot]
  table[row sep=crcr]{%
3.000	60.725\\
3.000	97.258\\
};
\addplot [color=black, forget plot]
  table[row sep=crcr]{%
0.887	86.616\\
1.113	86.616\\
};
\addplot [color=black, forget plot]
  table[row sep=crcr]{%
1.888	124.122\\
2.112	124.122\\
};
\addplot [color=black, forget plot]
  table[row sep=crcr]{%
2.888	158.591\\
3.112	158.591\\
};
\addplot [color=black, forget plot]
  table[row sep=crcr]{%
0.887	29.282\\
1.113	29.282\\
};
\addplot [color=black, forget plot]
  table[row sep=crcr]{%
1.888	43.861\\
2.112	43.861\\
};
\addplot [color=black, forget plot]
  table[row sep=crcr]{%
2.888	60.725\\
3.112	60.725\\
};
\addplot [color=blue, forget plot]
  table[row sep=crcr]{%
0.775	50.266\\
0.775	64.843\\
1.225	64.843\\
1.225	50.266\\
0.775	50.266\\
};
\addplot [color=blue, forget plot]
  table[row sep=crcr]{%
1.775	73.682\\
1.775	94.002\\
2.225	94.002\\
2.225	73.682\\
1.775	73.682\\
};
\addplot [color=blue, forget plot]
  table[row sep=crcr]{%
2.775	97.258\\
2.775	122.383\\
3.225	122.383\\
3.225	97.258\\
2.775	97.258\\
};
\addplot [color=red, forget plot]
  table[row sep=crcr]{%
0.775	57.644\\
1.225	57.644\\
};
\addplot [color=red, forget plot]
  table[row sep=crcr]{%
1.775	83.871\\
2.225	83.871\\
};
\addplot [color=red, forget plot]
  table[row sep=crcr]{%
2.775	109.979\\
3.225	109.979\\
};
\addplot [color=blue, draw=none, mark=+, mark options={solid, red}, forget plot]
  table[row sep=crcr]{%
1.000	23.230\\
1.000	87.911\\
1.000	89.587\\
1.000	90.426\\
1.000	91.161\\
1.000	92.212\\
1.000	92.717\\
1.000	95.371\\
1.000	104.042\\
};
\addplot [color=blue, draw=none, mark=+, mark options={solid, red}, forget plot]
  table[row sep=crcr]{%
2.000	41.288\\
2.000	124.661\\
2.000	127.435\\
2.000	128.896\\
2.000	131.183\\
2.000	135.048\\
2.000	136.054\\
};
\addplot [color=blue, draw=none, mark=+, mark options={solid, red}, forget plot]
  table[row sep=crcr]{%
3.000	161.070\\
3.000	162.856\\
3.000	163.316\\
3.000	166.774\\
3.000	167.664\\
3.000	169.684\\
3.000	170.813\\
3.000	171.888\\
};
\node[above, align=center, fill=white]
at (1cm,-0.55cm) {2};
\node[above, align=center, fill=white]
at (3cm,-0.55cm) {3};
\node[above, align=center, fill=white]
at (5cm,-0.55cm) {4};
\end{axis}
\end{tikzpicture}%
	\caption{Effect of the number of connections in $G$ on $\pi(S^*)$}
	\label{fig:effect_delta}
\end{figure}
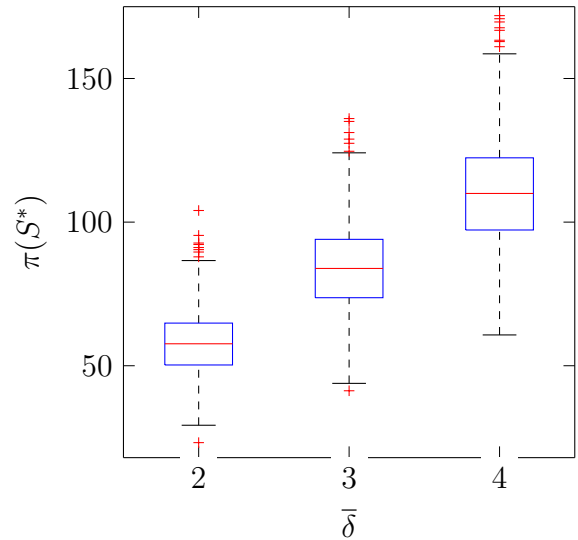

We observe that more connections in the connection graph $G$ result in cost increases, because we need to disconnect more edges in order to remove a LRU. This has an important managerial implication, as engineers should be urged to reduce the number of connections in systems to be developed. Thus, it may be wise for a company to invest extra in a system's design such that the number of connections in $G$ is reduced. An example wherein few number of connections lead to low costs is a bicycle. A typical connection graph of bicycles has few connections, and consequently relatively low replacement cost because we only need to disconnect few connections in case a part fails.

Finally, we also study the effect that system complexity has on the costs of the optimal solution $\pi(S^*)$, when we consider the number of predecessor--successor relationships. A lower value of $\delta_E$ indicates that fewer predecessor--successor relationships exist in the precedence graph $D$. Similar to the foregoing, we keep the other parameters constant at $\delta=3$ and $q=1$. The results for different values of $\delta_E$ are depicted in Figure \ref{fig:effect_delta_E} and we observe a similar behavior to changes in $\delta$.
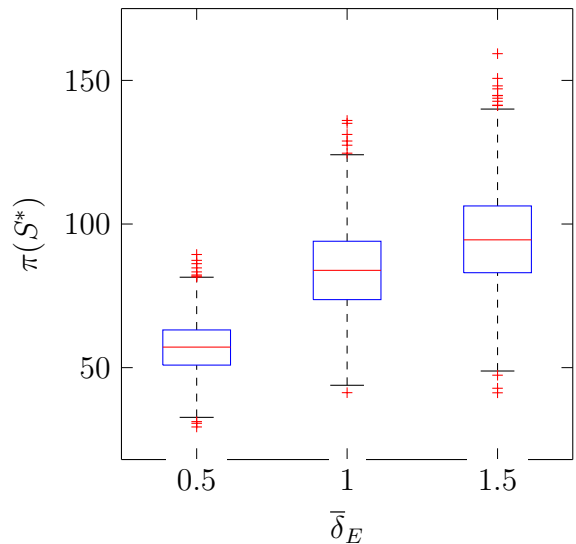
\begin{figure}[htbp!]
	\centering
	\setlength\figureheight{6cm} 
	\setlength\figurewidth{6cm}
%
%
\begin{tikzpicture}

\begin{axis}[%
width=\figurewidth,
height=\figureheight,
at={(0\figurewidth,0\figureheight)},
scale only axis,
clip=false,
separate axis lines,
every outer x axis line/.append style={black},
every x tick label/.append style={font=\color{black}},
every x tick/.append style={black},
xmin=0.5,
xmax=3.5,
xtick={1,2,3},
xlabel={$\overline{\delta}_E$},
every outer y axis line/.append style={black},
every y tick label/.append style={font=\color{black}},
every y tick/.append style={black},
ymin=18,
ymax=175,
ylabel={$\pi(S^*)$},
axis background/.style={fill=white}
]
\addplot [color=black, dashed, forget plot]
  table[row sep=crcr]{%
1.000	63.126\\
1.000	81.478\\
};
\addplot [color=black, dashed, forget plot]
  table[row sep=crcr]{%
2.000	94.002\\
2.000	124.122\\
};
\addplot [color=black, dashed, forget plot]
  table[row sep=crcr]{%
3.000	106.295\\
3.000	140.017\\
};
\addplot [color=black, dashed, forget plot]
  table[row sep=crcr]{%
1.000	32.691\\
1.000	50.884\\
};
\addplot [color=black, dashed, forget plot]
  table[row sep=crcr]{%
2.000	43.861\\
2.000	73.682\\
};
\addplot [color=black, dashed, forget plot]
  table[row sep=crcr]{%
3.000	48.827\\
3.000	83.051\\
};
\addplot [color=black, forget plot]
  table[row sep=crcr]{%
0.887	81.478\\
1.113	81.478\\
};
\addplot [color=black, forget plot]
  table[row sep=crcr]{%
1.888	124.122\\
2.112	124.122\\
};
\addplot [color=black, forget plot]
  table[row sep=crcr]{%
2.888	140.017\\
3.112	140.017\\
};
\addplot [color=black, forget plot]
  table[row sep=crcr]{%
0.887	32.691\\
1.113	32.691\\
};
\addplot [color=black, forget plot]
  table[row sep=crcr]{%
1.888	43.861\\
2.112	43.861\\
};
\addplot [color=black, forget plot]
  table[row sep=crcr]{%
2.888	48.827\\
3.112	48.827\\
};
\addplot [color=blue, forget plot]
  table[row sep=crcr]{%
0.775	50.884\\
0.775	63.126\\
1.225	63.126\\
1.225	50.884\\
0.775	50.884\\
};
\addplot [color=blue, forget plot]
  table[row sep=crcr]{%
1.775	73.682\\
1.775	94.002\\
2.225	94.002\\
2.225	73.682\\
1.775	73.682\\
};
\addplot [color=blue, forget plot]
  table[row sep=crcr]{%
2.775	83.051\\
2.775	106.295\\
3.225	106.295\\
3.225	83.051\\
2.775	83.051\\
};
\addplot [color=red, forget plot]
  table[row sep=crcr]{%
0.775	57.150\\
1.225	57.150\\
};
\addplot [color=red, forget plot]
  table[row sep=crcr]{%
1.775	83.871\\
2.225	83.871\\
};
\addplot [color=red, forget plot]
  table[row sep=crcr]{%
2.775	94.492\\
3.225	94.492\\
};
\addplot [color=blue, draw=none, mark=+, mark options={solid, red}, forget plot]
  table[row sep=crcr]{%
1.000	29.354\\
1.000	30.623\\
1.000	31.214\\
1.000	81.580\\
1.000	81.870\\
1.000	82.256\\
1.000	83.333\\
1.000	84.725\\
1.000	86.201\\
1.000	87.347\\
1.000	89.389\\
};
\addplot [color=blue, draw=none, mark=+, mark options={solid, red}, forget plot]
  table[row sep=crcr]{%
2.000	41.288\\
2.000	124.661\\
2.000	127.435\\
2.000	128.896\\
2.000	131.183\\
2.000	135.048\\
2.000	136.054\\
};
\addplot [color=blue, draw=none, mark=+, mark options={solid, red}, forget plot]
  table[row sep=crcr]{%
3.000	41.244\\
3.000	42.824\\
3.000	47.369\\
3.000	141.252\\
3.000	141.373\\
3.000	142.734\\
3.000	143.791\\
3.000	144.683\\
3.000	144.767\\
3.000	147.064\\
3.000	148.099\\
3.000	150.704\\
3.000	159.309\\
};
\node[above, align=center, fill=white]
at (1cm,-0.55cm) {0.5};
\node[above, align=center, fill=white]
at (3cm,-0.55cm) {1};
\node[above, align=center, fill=white]
at (5cm,-0.55cm) {1.5};
\end{axis}
\end{tikzpicture}%
	\caption{Effect of the number of predecessor--successor relationships in $D$ on $\pi(S^*)$}
	\label{fig:effect_delta_E}
\end{figure}

The costs increase as the number of predecessor--successor relationships increases, because we need to disconnect more connections upon the failure of a LRU. Consequently, the costs of an optimal solution $\pi(S^*)$ increase when the number of connections in $D$ increases (as $\delta_E$ increases). The managerial implications of our results also align with those for $\delta$: managers should urge their designers to avoid predecessor--successor relationships in order to reduce costs. This objective may be easier to attain than avoiding connections in the connection graph $G$ by careful design. Hence, the results confirm that careful design (in terms of $G$ and $D$) is crucial to reduce the overall costs.

Finally, we numerically illustrate the cost effect of introducing a partitioning constraint. Therefore, we consider \textsc{LRU Design} and the variant without a partition constraint called \textsc{C--LRU Design}. Further details on \textsc{C--LRU Design} are included in Appendix \ref{app:cover_LRU}. Both, \textsc{C--LRU Design} and the numerical results, are based on \citet{driessen2018life}. 
We study the average relative cost difference for a given parameter combination $\Delta_{\pi}=\frac{\pi(S^*)-\pi_c(S_c^*)}{\pi_c(S_c^*)}\times 100\%$, where $S^*$ and $S_c^*$ are the optimal solutions of \textsc{LRU Design} and \textsc{C--LRU Design}, respectively. We limit our analysis to instances with $|V|\in\{20,30,40\}$, because the set partitioning formulation of \textsc{LRU Design} is efficient for each parameter combination. 

\begin{table}[htbp!]
\centering
\begin{tabular}{crrrrrrr}
\hline
					&							&	\multicolumn{2}{c}{$|V|=20$}	&	\multicolumn{2}{c}{$|V|=30$}	&	\multicolumn{2}{c}{$|V|=40$}	\bigstrut[t]\\
\cline{3-4} \cline{5-6} \cline{7-8} 
$\delta$	&	$\delta_E$	&	avg (\%) 	& max (\%) & avg (\%) & max (\%) &	avg (\%)	& max (\%)\\
2     & 0.5   & 0.74 & 1.56		& 1.07	& 3.63	& 1.20	& 4.00	 \bigstrut[t]\\
2     & 1     & 1.62 & 4.36		& 2.23	& 5.49	& 2.52	& 4.20	\\
2     & 1.5   & 0.26 & 1.05		& 0.19	& 0.62	& 0.19	& 0.51	\\
3     & 0.5   & 0.56 & 1.36		& 0.63	& 1.73	& 0.54	& 1.11	\\
3     & 1     & 0.84 & 3.53		& 0.87	& 2.45	& 1.48	& 3.12	\\
3     & 1.5   & 2.37 & 13.56	& 0.56	& 2.56	& 0.36	& 2.34	\\
4     & 0.5   & 0.00 & 0.00		& 0.13	& 1.20	& 0.08	& 0.26	\\
4     & 1     & 1.58 & 6.07		& 1.02	& 3.78	& 0.47	& 1.39	\\
4     & 1.5   & 4.22 & 13.48	& 1.80	& 6.36	& 1.32	& 3.08	\\
All	  &       & 1.34 & 13.56	& 0.94	& 6.36	& 0.91	& 4.20 \\
\hline
\end{tabular}
\caption{Relative cost differences between \textsc{LRU Design} and \textsc{C--LRU Design}}
\label{tab:cost_delta}
\end{table}

The results in Table \ref{tab:cost_delta} illustrate that we sacrifice little costs when we introduce a partitioning constraint. Combining this with the practical benefits that \textsc{LRU Design} has over \textsc{C--LRU Design}, we conclude that \textsc{LRU Design} is highly usable for practical purposes.

\section{Conclusions}
\label{sec:conclusions}
We considered an OEM or MRO that is concerned with the (re)design and maintenance of a system. If the system does not operate, the company loses money, customer goodwill or has to pay customers a downtime penalty. Therefore, the company is interested in lowering the cost for non--functioning systems by designing Line Replaceable Units (LRUs) that can be removed quickly. Furthermore, the LRUs should not be too large, because this increases a LRU's total failure rate and the LRU's purchase cost (or repair cost). Thus, the company has to determine what the optimal LRU design is that balances the replacement cost and the purchase cost (or repair cost) of LRUs.

We presented a novel model for representing the connections between parts in a system, also capturing the existing disassembly sequences. 
We used the system representation to derive an optimization model \textsc{LRU Design} that minimizes the replacement cost and the purchase cost (or repair cost) by optimizing the LRU design. Our optimization model was constrained such that a LRU design is a partition of the parts, and we saw that this constraint has strong practical benefits (avoiding replacement ambiguity, simplification of the production and testing processes, and consistency with general maintenance practice), while sacrificing very little extra costs compared to the case where we relax the partition constraint. We formulated the problem as a binary linear program and as a set partitioning problem. We proved a result infrequently encountered in research: an optimal solution to the set partitioning formulation is integer, despite a non--integral feasible polyhedron. This result follows from proving the suboptimality of LRU cycles and relating this to the matrix encoding of an optimal solution. Furthermore, the set partitioning formulation reduces the computation times and makes it useful as a feedback mechanism to assess various design alternatives and their effects on the optimal LRU design and the corresponding (after--sales) costs. Moreover, the set partitioning formulation is suitable to solve large instances, while the binary linear programming formulation is not. In addition to the computation times, we observed that optimal solutions to \textsc{LRU Design} have larger LRUs when the cost per time unit of system downtime increases, because this enables faster replacement and thus avoids large downtime cost. Finally, we found that managers should urge their designers to reduce the number of connections and predecessor--successor relationships in a system's design.

\subsubsection*{Acknowledgement}
The authors would like to thank Willem van Jaarsveld for his suggestions and remarks that aided in establishing Theorem \ref{thm:integrality}. Moreover, we would like to express our gratitude to Gerhard Woeginger for better positioning the integrality result in the literature of combinatorial optimization. Finally, we thank the reviewers for their constructive feedback and one reviewer for the generalization to superadditive failure rate functions.

The authors also express their gratitude to ASML and Dutch Railways for their cooperation in this research. Furthermore, we gratefully acknowledge the support of The Netherlands Organisation for Scientific Research.

\appendix


\section{Deriving the successor collection $H(e)$}
\label{app:algorithms}
We determine $H(e)$ for all edges $e\in E$ in polynomial time by the following polynomial algorithms, where Algorithm \ref{alg:derive_degree} is called by Algorithm \ref{alg:derive_H}.

\begin{algorithm}[h!]
\caption{Derive $H(e)$ for all edges $e\in E$}\label{alg:derive_H}
\begin{algorithmic}[1]
\Procedure{RemovalEdges}{$E,A$}
\State $\hat{E}\gets E$
\State $\tilde{E}\gets \textsc{Degree}(D(\hat{E},A))$
\While{$\tilde{E}\neq \emptyset$}
	\ForAll{$e\in \tilde{E}$}
	\State $H(e)\gets \{e\}$
	\EndFor
	\ForAll{$e\in \tilde{E}$} 
		\ForAll{$(e,z)\in A$}
			\State $H(e)\gets H(e) \cup H(z)$
		\EndFor
	\EndFor
\State $\hat{E}\gets \hat{E}\setminus\tilde{E}$
\State $\tilde{E} \gets \textsc{Degree}(D(\hat{E},A))$
\EndWhile
\State\Return $H(e)$ for all $e\in E$
\EndProcedure
\end{algorithmic}
\end{algorithm}
\begin{algorithm}[h!]
\caption{Determine all edges $e\in E$ that have no successors in $D$}\label{alg:derive_degree}
\begin{algorithmic}[1]
\Procedure{Degree}{$D$}
\State $\tilde{E}\gets\emptyset$
\ForAll{$e\in E$}
	\If{$\delta^{out}(e)=0$}
		\State $\tilde{E} \gets \tilde{E}\cup \{e\}$
	\EndIf
\EndFor
\State\Return $\tilde{E}$
\EndProcedure
\end{algorithmic}
\end{algorithm}

\section{LRU with internal links that must be broken}
\label{sec:app:chain}
The precedence graph determines which edges must {\em necessarily} be broken before a certain edge can be broken. Therefore it is possible to define a LRU $Q$ such that certain connections within LRU $Q$ must be broken in order to be able to detach LRU $Q$ from the rest of the system. While such situations do not happen often in optimal designs in practice, it is sometime inevitable. Why this cannot physically be avoided is illustrated by a simple example. 

Consider two cog-wheels with a chain between them such that motion in one cog is transmitted to the other cog. Every bicycle has such a system to transmit motion from the pedals to the rear-wheel and a picture of such a system is shown below in Figure \ref{fig:tandwielketting}. The chain consists of many links. Now one may wish to define the chain as a LRU. To remove the chain from the two-cog-wheels on which it sits, one must necessarily break the connection between two links in the chain, despite the fact that such a connection is internal to the LRU/chain.

This example can be formalized in the connection and precedence graph shown in Figures \ref{fig:connectiontandwielketting} and \ref{fig:precedencetandwielketting} respectively. In these graphs $A$ and $B$ are the cog-wheels and nodes 1 to 12 correspond to the links in the chain. In this example the chain LRU is given by $Q=\{1,2,\ldots,12\}$ so that
$\Gamma(Q)=\{\{A,3\},\{B,3\},\{3,4\}\}$. The fact that that $\{3,4\}$ is internal to chain reflects the fact that it is physically impossible to remove the chain from the cog-wheels without first breaking a connection between two links in the chain. Thus the precedence graph models the physical necessity of breaking certain connections before it is possible to break/access other connections. The fact that $\Gamma(Q)$ may contain connections that are internal to an LRU $Q$ is a modeling feature that allows us to model systems appropriately. However, we do observe that most LRUs in optimal designs do not have internal connections that need to be broken.

\begin{figure}[h!]
\begin{center}
\includegraphics[width=0.3\textwidth]{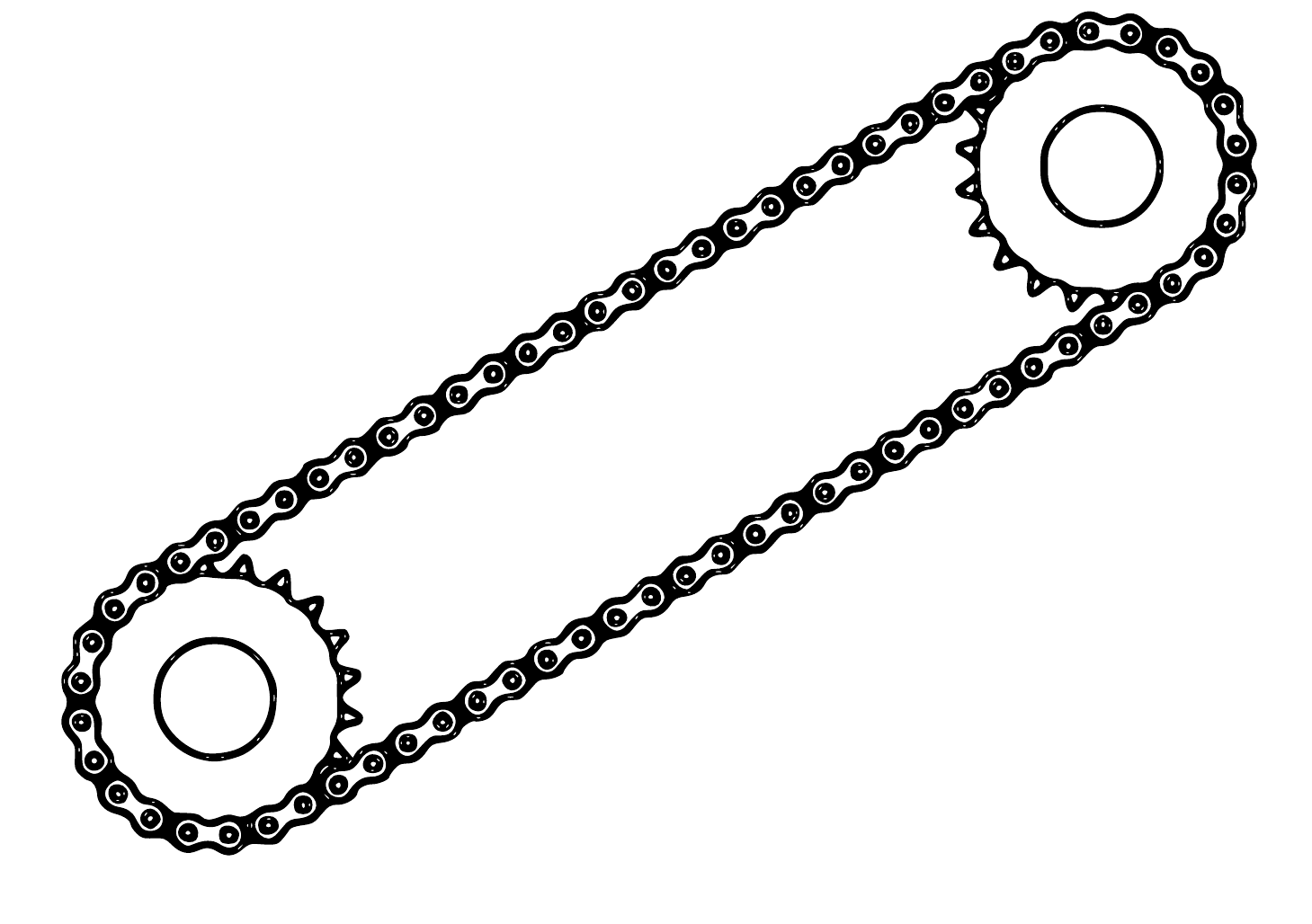}
\end{center}
\caption{Two cog-wheels connected with a chain}
\label{fig:tandwielketting}
\end{figure}

\begin{figure}[h!]
\begin{center}
\begin{tikzpicture}[main/.style = {draw, circle}] 
\node[main] (A) {A}; 
\node[main] (1) [above of=A] {1};
\node[main] (2) [right of=1] {2}; 
\node[main] (3) [right of=2] {3};
\node[main] (4) [right of=3] {4};
\node[main] (5) [right of=4] {5};
\node[main] (B) [below of=5] {B};
\node[main] (6) [right of=B] {6};
\node[main] (7) [below of=B] {7};
\node[main] (8) [left of=7] {8};
\node[main] (9) [left of=8] {9};
\node[main] (10) [left of=9] {10};
\node[main] (11) [left of=10] {11};
\node[main] (12) [left of=A] {12};
\draw (1) -- (2);
\draw (2) -- (3);
\draw (3) -- (4);
\draw (4) -- (5);
\draw (5) -- (6);
\draw (6) -- (7);
\draw (7) -- (8);
\draw (8) -- (9);
\draw (9) -- (10);
\draw (10) -- (11);
\draw (11) -- (12);
\draw (12) -- (1);
\draw (A) -- (3);
\draw (B) -- (3);
\end{tikzpicture} 
\end{center}
\caption{Connection graph for two cog-wheels with a chain}
\label{fig:connectiontandwielketting}
\end{figure}
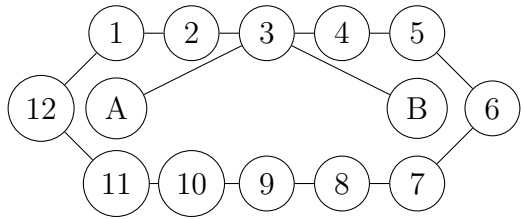

\begin{figure}[h!]
\begin{center}
\begin{tikzpicture}[node distance={20mm}, main/.style = {draw, circle}] 
\node[main] (A3) {$\{A,3\}$}; 
\node[main] (34) [right of=A3] {$\{3,4\}$};
\node[main] (B3) [right of=34] {$\{B,3\}$}; 
\draw[->] (A3) -- (34);
\draw[->] (B3) -- (34);
\end{tikzpicture} 
\end{center}
\caption{Precedence graph for two cog-wheels with a chain}
\label{fig:precedencetandwielketting}
\end{figure}
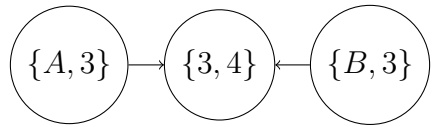

%
\section{LRU Design without partitioning constraint}
\label{app:cover_LRU}
In this appendix, we explain the model of designing LRUs without considering a partitioning constraint. We call this problem \textsc{C--LRU Design}. Let a system be represented by a connection graph $G=(V,E)$ and a precedence graph $D=(E,A)$, with $V$, $E$ and $A$ corresponding to the vertex set, edge set and arc set, respectively. Furthermore, we let $\lambda(v)>0$ and $\ell(v)>0$ be the failure rate and purchase cost of part $v\in V$, and we define $w(e)>0$ as the cost of breaking edge $e\in E$. An arc $(\{u,v\},\{v,x\})\in A$ denotes that we have to break $\{v,x\}$ prior to breaking $\{u,v\}$.

A part $v\in V$ belongs to at least one LRU. That is, part $v$ is replaced upon its own failure, but it may also be replaced upon the failure of a part $u\in V:u\neq v$. Therefore, we represent a LRU differently from \textsc{LRU Design}. We let a LRU be a tuple characterized by a replacement set and a failure set, i.e., $Q=(R_Q,F_Q)$ where $Q$ is the LRU, $R_Q\subseteq V$ is the replacement set, and $F_Q\subseteq R_Q$ is the failure set. The failure of a part in the failure set 
triggers replacement of the LRU. The replacement set is 
replaced if any of the vertices in the failure set fails. 
Furthermore, we assume that a part $v\in V$ belongs to exactly one failure set, i.e., the failure sets partition $V$. 

Next, we study what happens when a LRU $Q$ 	fails, or technically what happens when a part $v\in F_Q$ fails. In this case, we have to break all edges $e\in\Gamma(R_Q)$, where $\Gamma(R_Q)=\bigcup_{e\in B(R_Q)}H(e)$. The failure rate of LRU $Q$ is given by $\sum_{v\in F_Q}\lambda(v)$, because all parts $v\in F_Q$ induce the replacement of $R_Q$. 
Similarly, the total purchase cost of LRU $Q$ is given by $\sum_{v\in R_Q}\ell (v)$. This yields the following average cost per time unit for LRU $Q$: 
$$\omega_c(Q)=\sum_{e\in \Gamma(R_Q)}w(e)\sum_{v\in F_Q}\lambda(v)+\sum_{u\in R_Q}\ell (u)\sum_{v\in F_Q}\lambda(v).$$

We are interested in determining the optimal LRU design. Let $S_c$ be a collection of LRUs such that $\emptyset\not\in S_c$ and each part $v\in V$ is included in at least one replacement set and in exactly one failure set; i.e., $\bigcup_{Q\in S_c}F_Q=V$, $F_{Q}\cap F_{Q^{\prime}}=\emptyset$ for all $Q, Q^{\prime}\in S_c:F_{Q}\neq F_{Q^{\prime}}$, and $F_Q\subseteq R_Q$ for each LRU $Q\in S_c$. The average costs per time unit of a LRU design $S_c$ are given by
\begin{equation}
\begin{split}
\pi_c(S_c)&=\sum_{Q\in S_c}\omega_c(Q)=\sum_{Q\in S_c}\sum_{b\in \Gamma(R_Q)}w(b)\sum_{v\in F_Q}\lambda(v)+\sum_{Q\in S_c}\sum_{u\in R_Q}\ell (u)\sum_{v\in F_Q}\lambda(v).
\end{split}
\label{eq:cost_cover}
\end{equation}

Next, we define \textsc{C--LRU design} as: What is the LRU Design $S_c$ that minimizes $\pi_c(S_c)$?

For implementation, we use the binary programming formulation of \textsc{C--LRU Design} as explained in \citet{driessen2018life}, which is similar to the binary programming formulation of \textsc{LRU Design}.

\section{Lemma used in the proof of Theorem \ref{thm:intersection_cycle}}
\label{app:lemma_removal_path}

The proof of Theorem \ref{thm:intersection_cycle}, which states that an optimal solution to LPM does not contain a LRU cycle, uses Lemma \ref{lem:removal_path}. This appendix contains the statement, an interpretation and the proof of that lemma. Recall that we define the set of edges that is broken for a LRU $X$ but not for a LRU $Y$ by $\mathcal{F}(X,Y)=\Gamma(X)\setminus\Gamma(Y)$. Furthermore, recall that we use modular arithmetic for the indices of LRUs that form a cycle $Q_i$, $Q_{i-1}$, and $Q_{i+1}$ with $1\leq i\leq n$, $n+1\equiv 1 \pmod{n}$, and $Q_0\equiv Q_n$.

\begin{lemma}
Given a solution $x$ to LPM that contains a LRU cycle $C=\{Q_1,Q_2,\ldots,Q_n\}\subseteq S$ with minimal $n$, we have $\Gamma(Q_i\cap Q_{i+1})\setminus \mathcal{F}(Q_i\cap Q_{i+1},Q_i)\subseteq \Gamma(Q_i)\setminus\mathcal{F}(Q_i\cap Q_{i-1},Q_{i-1})$ and $\Gamma(Q_i\setminus Q_{i+1})\setminus \mathcal{F}(Q_i\cap Q_{i+1},Q_i)\subseteq \Gamma(Q_i)\setminus \mathcal{F}(Q_i\cap Q_{i+1},Q_{i+1})$. Furthermore, we have $\Gamma(Q_i\cap Q_{i-1})\setminus \mathcal{F}(Q_i\cap Q_{i-1},Q_i)\subseteq \Gamma(Q_i)\setminus\mathcal{F}(Q_i\cap Q_{i+1},Q_{i+1})$ and  $\Gamma(Q_i\setminus Q_{i-1})\setminus \mathcal{F}(Q_i\cap Q_{i-1},Q_i)\subseteq \Gamma(Q_i)\setminus \mathcal{F}(Q_i\cap Q_{i-1},Q_{i-1})$.
\label{lem:removal_path}
\end{lemma}

Lemma \ref{lem:removal_path} can be interpreted by considering Figure \ref{fig:illustrate_lemma2}, which is a simplified part of a solution in which we depict only the LRUs $Q_{i-1}$, $Q_i$, and $Q_{i+1}$, and we assume that $A=\emptyset$. The first claim in Lemma \ref{lem:removal_path} states that the set of edges that are broken for LRU $Q_i$ as well as for the intersection $Q_i\cap Q_{i+1}$ (the edges that cross the thin dotted line in Figure \ref{fig:illustrate_lemma2}) is a subset of the edges that are broken for $Q_i$ except for the edges that are broken for $Q_i\cap Q_{i-1}$ but not for $Q_{i-1}$ (the edges that cross the thin solid line in Figure \ref{fig:illustrate_lemma2}). Secondly, the set of edges that is broken for both $Q_i$ and $Q_i\setminus Q_{i+1}$ (the edges that cross the thick dotted line in Figure \ref{fig:illustrate_lemma2}) is a subset of the edges that are broken for $Q_i$ except for the edges that are broken for $Q_i\cap Q_{i+1}$ but not for $Q_{i+1}$ (the edges that cross the thick solid line in Figure \ref{fig:illustrate_lemma2}). 
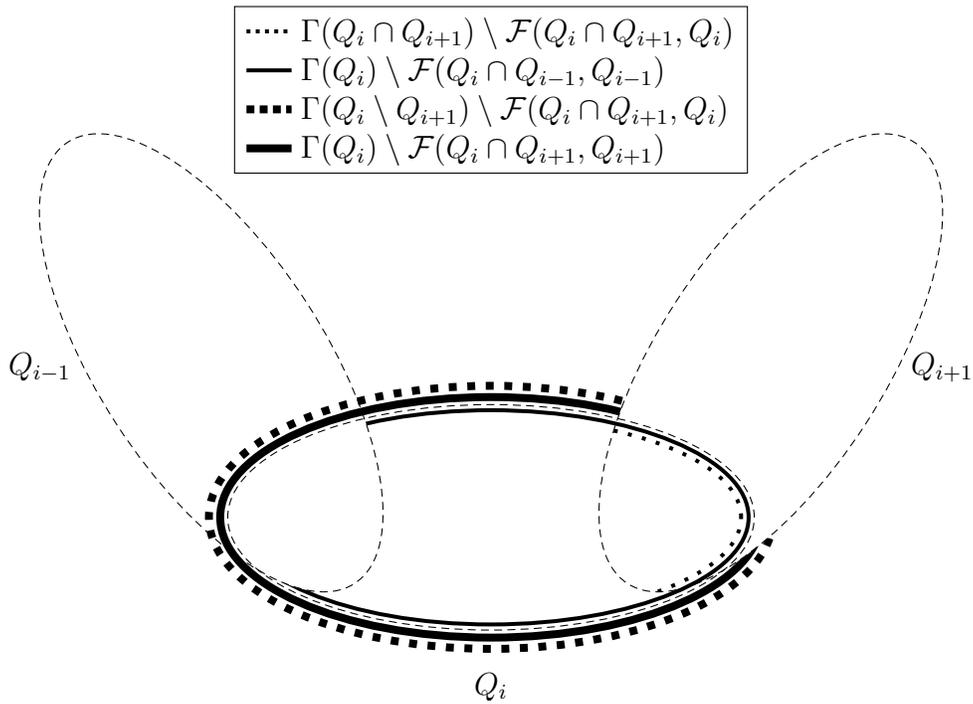
\begin{figure}[htbp!]
\centering
\begin{tikzpicture}[scale=1,transform shape]
\tikzset{LabelStyle/.style = {draw=none,fill=none,text=black}}
\tikzset{EdgeStyle/.style   = {draw,fill=none}}
	\clip (-7,-2.5) rectangle (7,7);
	\tikzset{VertexStyle/.style = {circle,draw=none,fill=white}}
	\node[VertexStyle] (100) at (0,-2.25) {$Q_i$};
	\node[VertexStyle] (101) at (-6,2) {$Q_{i-1}$};
	\node[VertexStyle] (102) at (6,2) {$Q_{i+1}$};

 	\draw[densely dashed] (0,0) ellipse (3.5 and 1.5);
	\draw[densely dashed,rotate=-33] (2,3.75) ellipse (1.5 and 3.5);
	\draw[densely dashed,rotate=33] (-2,3.75) ellipse (1.5 and 3.5);
	
	\begin{scope}
		\clip[rotate=-33] (2,3.75) ellipse (1.5 and 3.5);
  		\draw[loosely dotted,line width=1.5pt,draw=black] (0,0) ellipse (3.325 and 1.325);
	\end{scope}
	
	\begin{scope}
	\clip[rotate=33] (-3.5,-3.5) rectangle (3.5,3.5)
      	  (-2,3.75) ellipse (1.5 and 3.5);
	\draw[line width=1.5pt,draw=black] (0,0) ellipse (3.425 and 1.425);
	\end{scope}
	
	\begin{scope}
	\clip[rotate=-33] (-3.5,-3.5) rectangle (3.5,3.5)
      	  (2,3.75) ellipse (1.5 and 3.5);
	\draw[line width=3pt,draw=black] (0,0) ellipse (3.6 and 1.6);
	\draw[loosely dotted,line width=3pt,draw=black] (0,0) ellipse (3.75 and 1.75);
	\end{scope}

	\begin{customlegend}[legend cell align=left,
	legend entries={
	$\Gamma(Q_i\cap Q_{i+1})\setminus {\mathcal{F}(Q_i\cap Q_{i+1},Q_i)}$,
	$\Gamma(Q_i)\setminus{\mathcal{F}(Q_i\cap Q_{i-1},Q_{i-1})}$,
	$\Gamma(Q_i\setminus Q_{i+1})\setminus {\mathcal{F}(Q_i\cap Q_{i+1},Q_i)}$,
	$\Gamma(Q_i)\setminus {\mathcal{F}(Q_i\cap Q_{i+1},Q_{i+1})}$},
	legend style={at={(0,5)},anchor=mid}]
    \addlegendimage{black,line width=1.5pt,dotted,sharp plot}
    \addlegendimage{black,line width=1.5pt,sharp plot}
    \addlegendimage{black,line width=3pt,dotted,sharp plot}
    \addlegendimage{black,line width=3pt,sharp plot}
    \end{customlegend}

\end{tikzpicture}
\caption{Interpretation of Lemma \ref{lem:removal_path}. The edges that cross the lines shown in the legend are contained in the corresponding sets of the legend.}
\label{fig:illustrate_lemma2}
\end{figure}
The third and fourth claim from Lemma \ref{lem:removal_path} can be interpreted similarly to the first two claims, where $Q_{i+1}$ is replaced by $Q_{i-1}$ and vice versa.

The same interpretation -- as illustrated in Figure \ref{fig:illustrate_lemma2} -- holds when $A\neq\emptyset$, but this makes exposition cumbersome. The proof of Lemma \ref{lem:removal_path} is given below, for arbitrary $A$.

\begin{proof}[Proof of Lemma \ref{lem:removal_path}.]
Let $x$ be a solution to LPM that contains a LRU cycle $C=\{Q_1,Q_2,\ldots,Q_n\}\subseteq S$ with minimal $n$. First, we prove that if $e=\{u,v\}\in\mathcal{F}(Q_i\cap Q_{i-1},Q_{i-1})$ then $u,v \in Q_{i-1}$. Subsequently, we use this result to show that $\Gamma(Q_i\cap Q_{i+1})\setminus \mathcal{F}(Q_i\cap Q_{i+1},Q_i)\subseteq \Gamma(Q_i)\setminus\mathcal{F}(Q_i\cap Q_{i-1},Q_{i-1})$. Similarly, we show that if $e=\{u,v\}\in\mathcal{F}(Q_i\cap Q_{i+1},Q_{i+1})$ then $u,v\in Q_{i+1}$, and we use this to prove that $\Gamma(Q_i\setminus Q_{i+1})\setminus\mathcal{F}(Q_i\cap Q_{i+1},Q_i)\subseteq \Gamma(Q_i)\setminus\mathcal{F}(Q_i\cap Q_{i+1},Q_{i+1})$.

Let $e=\{u,v\}\in\mathcal{F}(Q_i\cap Q_{i-1},Q_{i-1})$. If $u\in Q_{i-1}$ and $v\not\in Q_{i-1}$, it follows directly that $e\in \Gamma(Q_{i-1})$ as $e\in B(Q_{i-1})$. This contradicts that $e\in\mathcal{F}(Q_i\cap Q_{i-1},Q_{i-1})$. If $u,v\not\in Q_{i-1}$, then there exists a (sub)path $(e_1,e_2,\ldots,e)$ in the precedence graph with $e_1=\{u_1,v_1\}: u_1\in Q_i\cap Q_{i-1}$. Consequently, there is an $e_j=\{u_j,v_j\}:u_j\in Q_{i-1}, v_j \not\in Q_{i-1}$ in this path, due to the assumption that $(\{\overline{u},\overline{v}\},\{\overline{v},\overline{x}\})\in A$ for $\overline{u},\overline{v},\overline{x}\in V$. Therefore, $e_j\in \Gamma(Q_{i-1})$ and thus $e\in\Gamma(Q_{i-1})$. Subsequently, $e\not\in\mathcal{F}(Q_i\cap Q_{i-1},Q_{i-1})$, which is a contradiction. Therefore, if $e=\{u,v\}\in\mathcal{F}(Q_i\cap Q_{i-1},Q_{i-1})$, then $u,v\in Q_{i-1}$.

Next, we show that $\Gamma(Q_i\cap Q_{i+1})\setminus \mathcal{F}(Q_i\cap Q_{i+1},Q_i)\subseteq \Gamma(Q_i)\setminus\mathcal{F}(Q_i\cap Q_{i-1},Q_{i-1})$. Let $e=\{u,v\}\in \Gamma(Q_i\cap Q_{i+1})$ and suppose that $e\in\mathcal{F}(Q_i\cap Q_{i-1},Q_{i-1})$. There exists a (sub)path $(e_1,e_2,\ldots,e)$ in the precedence graph with $e_1=\{u_1,v_1\}:u_1 \in Q_i\cap Q_{i+1}$. Also, $u_1\not\in Q_{i-1}$, because $C$ is a LRU cycle with minimal $n$ (implying that $Q_{i-1}\cap Q_i\cap Q_{i+1}=\emptyset$, otherwise one can omit $Q_i$ and obtain a LRU cycle with fewer LRUs, which is a contradiction). Furthermore, $u,v\in Q_{i-1}$ as $e=\{u,v\}\in\mathcal{F}(Q_i\cap Q_{i-1},Q_{i-1})$ (see above). Hence, there exists $e_j=\{u_j,v_j\}:u_j\in Q_{i-1},v_j\not\in Q_{i-1}$ in this path, as $(\{\overline{u},\overline{v}\},\{\overline{v},\overline{x}\})\in A$ for $\overline{u},\overline{v},\overline{x}\in V$. Therefore, $e_j\in\Gamma(Q_{i-1})$ and thus $e\in\Gamma(Q_{i-1})$. Consequently, $e\not\in\mathcal{F}(Q_i\cap Q_{i-1},Q_{i-1})$ which is a contradiction. Hence, if $e\in\Gamma(Q_i\cap Q_{i+1})$, then $e\not\in\mathcal{F}(Q_i\cap Q_{i-1},Q_{i-1})$. 
Next, we consider $\Gamma(Q_i\cap Q_{i+1})\setminus \mathcal{F}(Q_i\cap Q_{i+1},Q_i)$, and have that $\Gamma(Q_i\cap Q_{i+1})\setminus \mathcal{F}(Q_i\cap Q_{i+1},Q_i)\subseteq \Gamma(Q_i)$ as $\mathcal{F}(Q_i\cap Q_{i+1},Q_i)=\Gamma(Q_i\cap Q_{i+1})\setminus \Gamma(Q_i)$. Furthermore, each $e\in \Gamma(Q_i\cap Q_{i+1})\setminus \mathcal{F}(Q_i\cap Q_{i+1},Q_i)$ satisfies $e\in \Gamma(Q_i\cap Q_{i+1})$ and thus $e\not\in \mathcal{F}(Q_i\cap Q_{i-1},Q_{i-1})$, as we proved in the foregoing. Hence, we have $\Gamma(Q_i\cap Q_{i+1})\setminus \mathcal{F}(Q_i\cap Q_{i+1},Q_i)\subseteq \Gamma(Q_i)\setminus\mathcal{F}(Q_i\cap Q_{i-1},Q_{i-1})$.

Let us now prove that if $e=\{u,v\}\in\mathcal{F}(Q_i\cap Q_{i+1},Q_{i+1})$, then $u,v\in Q_{i+1}$. The proof is identical to proving that if $e\in\mathcal{F}(Q_i\cap Q_{i-1},Q_{i-1})$ then $u,v\in Q_{i-1}$, but $Q_{i-1}$ is replaced by $Q_{i+1}$. 
We continue by proving that $\Gamma(Q_i\setminus Q_{i+1})\setminus\mathcal{F}(Q_i\cap Q_{i+1},Q_i)\subseteq \Gamma(Q_i)\setminus\mathcal{F}(Q_i\cap Q_{i+1},Q_{i+1})$. Let $e=\{u,v\}\in\Gamma(Q_i\setminus Q_{i+1})$ and suppose that $e\in\mathcal{F}(Q_i\cap Q_{i+1},Q_{i+1})$. There exists a (sub)path $(e_1,e_2,\ldots,e)$ in the precedence graph with $e_1=\{u_1,v_1\}:u_1\in Q_i\setminus Q_{i+1}$. Moreover, $u,v\in Q_{i+1}$ as $e=\{u,v\}\in\mathcal{F}(Q_i\cap Q_{i+1},Q_{i+1})$ (from foregoing). Hence, there exists $e_j=\{u_j,v_j\}:u_j\in Q_{i+1},v_j\not\in Q_{i+1}$ in this path, as $(\{\overline{u},\overline{v}\},\{\overline{v},\overline{x}\})\in A$ for $\overline{u},\overline{v},\overline{x}\in V$. Therefore, $e_j\in\Gamma(Q_{i+1})$ and thus $e\in \Gamma(Q_{i+1})$. As a consequence, $e\not\in\mathcal{F}(Q_i\cap Q_{i+1},Q_{i+1})$, which is a contradiction. Hence, if $e\in \Gamma(Q_i\setminus Q_{i+1})$, then $e\not\in\mathcal{F}(Q_i\cap Q_{i+1},Q_{i+1})$. 
Next, we consider $\Gamma(Q_i\setminus Q_{i+1})\setminus\mathcal{F}(Q_i \cap Q_{i+1},Q_i)$. We observe that $\mathcal{F}(Q_i\cap Q_{i+1},Q_i)=\mathcal{F}(Q_i\setminus Q_{i+1},Q_i)$, because $Q_i\cap Q_{i+1}$ and $Q_i\setminus Q_{i+1}$ partition $Q_i$, and thus the edges existing between $Q_i\cap Q_{i+1}$ and $Q_i\setminus Q_{i+1}$ are the same. Hence, $\Gamma(Q_i\setminus Q_{i+1})\setminus\mathcal{F}(Q_i \cap Q_{i+1},Q_i)=\Gamma(Q_i\setminus Q_{i+1})\setminus\mathcal{F}(Q_i \setminus Q_{i+1},Q_i)\subseteq \Gamma(Q_i)$, where the last inequality holds as $\mathcal{F}(Q_i\setminus Q_{i+1},Q_i)=\Gamma(Q_i\setminus Q_{i+1})\setminus \Gamma(Q_i)$. Finally, we have $\Gamma(Q_i\setminus Q_{i+1})\setminus\mathcal{F}(Q_i\cap Q_{i+1},Q_i)\subseteq \Gamma(Q_i)\setminus\mathcal{F}(Q_i\cap Q_{i+1},Q_{i+1})$, since each $e\in \Gamma(Q_i\setminus Q_{i+1})$ is such that $e\not\in \mathcal{F}(Q_i\cap Q_{i+1},Q_{i+1})$.

Finally, we can do the same analysis as above, but replace $Q_{i-1}$ by $Q_{i+1}$ and vice versa. Then, we obtain $\Gamma(Q_i\cap Q_{i-1})\setminus \mathcal{F}(Q_i\cap Q_{i-1},Q_i)\subseteq \Gamma(Q_i)\setminus\mathcal{F}(Q_i\cap Q_{i+1},Q_{i+1})$ and  $\Gamma(Q_i\setminus Q_{i-1})\setminus \mathcal{F}(Q_i\cap Q_{i-1},Q_i)\subseteq \Gamma(Q_i)\setminus \mathcal{F}(Q_i\cap Q_{i-1},Q_{i-1})$. 
\end{proof}

\bibliography{references.bib} 

\end{document}